\documentclass[hidelinks,onefignum,onetabnum]{siamart250211}

\usepackage{lipsum}
\usepackage{amsfonts}
\usepackage{graphicx}
\usepackage{enumitem}
\usepackage{booktabs}
\usepackage{tikz}
\usetikzlibrary{arrows.meta, positioning}
\usepackage{pgfplots}
\usetikzlibrary{patterns}
\usepackage{amssymb}
\usepackage{xcolor}
\usepackage{cancel}
\usepackage{comment}
\usepackage{subfig}
\usepackage{subcaption}
\usepackage{epstopdf}
\usepackage{algorithmic}
\ifpdf
  \DeclareGraphicsExtensions{.eps,.pdf,.png,.jpg}
\else
  \DeclareGraphicsExtensions{.eps}
\fi

\newsiamremark{remark}{Remark}
\newsiamremark{hypothesis}{Hypothesis}
\crefname{hypothesis}{Hypothesis}{Hypotheses}
\newsiamthm{claim}{Claim}
\newsiamremark{fact}{Fact}
\crefname{fact}{Fact}{Facts}
\newsiamremark{example}{Example}
\headers{Asymptotic convergence analysis of HiPPA}{M. Ahookhosh, A. Iusem, A. Kabgani and F. Lara}

\title{Asymptotic convergence analysis of high-order proximal-point methods beyond sublinear rates \thanks{Submitted to the editors DATE.
\funding{MA and AK were partially supported by the Research Foundation Flanders (FWO) research project G081222N, UA BOF DocPRO4 project with ID 46929, and FWO travel grants K153324N and K152524N. FL was partially supported by ANID--Chile under project Fondecyt Regular 1241040 and by ECOS-ANID project ECOS240031.}}}

\author{
Masoud Ahookhosh\thanks{Department of Mathematics, University of Antwerp, Antwerp, Belgium.
 \\ (\email{masoud.ahookhosh@uantwerp.be}, \email{alireza.kabgani@uantwerp.be}).}
 \and Alfredo Iusem\thanks{Escola de Matemática Aplicada (EMAp), Fundação Getúlio Vargas (FGV), Rio de Janeiro, Brazil.
  (\email{alfredo.iusem@fgv.br}).}
 \and Alireza Kabgani\footnotemark[2]
 \and Felipe Lara\thanks{Instituto de Alta Investigación (IAI), Universidad de Tarapacá, Arica, Chile.
 \\ (\email{flarao@academicos.uta.cl}).} 
}

\usepackage{amsopn}

\ifpdf
\hypersetup{
  pdftitle={Asymptotic convergence analysis of high-order proximal-point methods},
  pdfauthor={Masoud Ahookhosh and Alfredo Iusem and Alireza Kabgani and Felipe Lara}
}
\fi

\newcommand\R{\mathbb{R}}
\newcommand\Rinf{\overline{\mathbb{R}}}
\newcommand\inter[1]{ {\rm \textbf{int}}(#1)} 
\newcommand\dom[1]{ \bs{{\rm dom}}(#1)} 
\newcommand\gf{\varphi} 
\newcommand\gh{\psi} 
\newcommand\fgam[3]{#1_{#3}^{#2}}

\newcommand\prox[3]{ \bs{{\rm prox}}_{#2#1}^{#3}}

\newcommand\ov[1]{\overline{#1}}
\newcommand\mb{\mathbf{B}}
\newcommand\bs[1]{\boldsymbol{#1}}
\newcommand\argmin[1]{\bs{\arg\min}_{#1}}
\newcommand\argmint[1]{\mathop{\bs{\arg\min}}\limits_{#1}}
\newcommand\Nz{\mathbb{N}_0}

\begin{document}

\maketitle
\vspace{-2mm}
\begin{abstract}
This paper investigates the asymptotic convergence behavior of the high-order proximal-point algorithm (HiPPA) to global minimizers, extending existing analyses beyond sublinear convergence rates and complexity analysis. Specifically, we study the proximal operator of a proper lower semicontinuous function augmented with a $p$th-order regularization for $p>1$, and establish the convergence of HiPPA to a global minimizer with a particular focus on its convergence rate. To this end, we focus on minimizing functions in the class of uniformly quasiconvex functions, which includes strongly convex, uniformly convex, and strongly quasiconvex functions as special cases.  Our analysis reveals the following convergence behaviors of HiPPA when the uniform quasiconvexity modulus $\phi$ admits a power function of degree $q$ as a lower bound, i.e., $\phi(t) \geq c t^q$ for some $c>0$, on an interval $\mathcal{I}$: (i) for $q\in (1,2)$ and $\mathcal{I}=[0,1)$, HiPPA exhibits a local linear rate for $p\in [q,2)$;  
(ii) HiPPA converges linearly when $p=2$, $q=2$, and also when $p=q>2$, provided that 
$\mathcal{I}=[0,\infty)$;
(iii) for $q\geq 2$ and $\mathcal{I}=[0,\infty)$, HiPPA achieves a superlinear rate for $p>q$.
Notably, to our knowledge, some of these results are novel, even in the context of strongly or uniformly convex functions, offering new insights into optimizing generalized convex problems.
\end{abstract}

\begin{keywords}
nonconvex optimization, uniformly quasiconvex functions, high-order Moreau envelope, high-order proximal method, linear and superlinear convergence rates 
\end{keywords}

\begin{MSCcodes}
49J52, 65K05, 90C26, 90C56
\end{MSCcodes}

\vspace{-2mm}
\section{Introduction}\label{intro}
In this paper, we consider the minimization problem 
\begin{equation}\label{eq:mainproblem}
\mathop{\bs\min}\limits_{x\in C}\ \gf(x),
\end{equation}
where $\gf: \R^n \to \Rinf := \R \cup \{+\infty\}$ is a proper and lower semicontinuous function, and  $C \subseteq \dom{\gf}$ is a nonempty closed convex set. We aim to minimize this function using the high-order proximal-point algorithm (HiPPA) (see Algorithm~\ref{alg:HiPPA} for more details)
\begin{equation}\label{eq:hippa}\tag{HiPPA}
x^{k+1}\in\prox{\gf}{\gamma_k}{p} (C,x^k):=\argmint{y\in C} \left(\gf(y)+\frac{1}{p\gamma_k}\Vert x^k- y\Vert^p\right),
\end{equation}
for $p>1$ and $\gamma_k\geq \gamma_{\bs\min}>0$, which includes the classical proximal-point algorithm (PPA) with $p=2$ as a special case; see, e.g., \cite{martinet1970breve,martinet1972determination,rockafellar1976monotone}. Due to the remarkable theoretical and computational properties of PPA, such as simplicity and low memory requirements, it has received much attention during the last few decades in both convex and nonconvex settings (e.g., \cite{Bauschke17,Guler92,KecisThibault15,Parikh14,Poliquin96,rockafellar1976monotone,Stella17,Themelis18}). 
Thanks to these desirable features, combined with its close connection to splitting algorithms and its key role in smoothing methods (e.g., Moreau envelope \cite{Moreau65} and forward-backward envelope \cite{Ahookhosh21,themelis2019acceleration,Themelis18}), it has become an unavoidable tool in the field of nonsmooth optimization; cf. \cite{Asi2019Stochastic,Bredies2022Degenerate,iusem2018second,Iusem2022proximal,Iusem20204two-step,Milzarek2024Semismooth}. 

It is worthwhile to note that PPA converges to a global minimizer in convex settings, whereas in nonconvex settings, convergence to a 
stationary
point, either a local minimizer or a saddle point, can only be guaranteed. In recent years, however, there has been growing interest in developing optimization methods that are capable of avoiding saddle points; see, e.g., \cite{Davis2022Escaping,davis2022proximal,lee2019first,vlatakis2019efficiently}, which is closely related to the identification of benign optimization landscapes in many structured problems arising in various application domains; see, e.g., \cite{Josz2023Certifying,ma2022blessing,McRae2024Benign,Vidal_Zhu_Haeffele_2022} and references therein. As an alternative approach, one can model the problem of interest using special classes,
such as convex or strongly quasiconvex functions, that inherently exclude saddle points; cf. \cite{aujol2024fista,grimmer2025some,iusem2018second,Iusem2022proximal,Iusem20204two-step,Kabgani2023,Lara2022strongly}. This motivates further exploration of generalized convexity classes, excluding saddle points, and ensures convergence to local or global minimizers. Section~\ref{sec:ben} presents one such class: uniformly quasiconvex functions.
We recall that a proper function $\gh:\R^n\to\Rinf$ with a convex domain
is said to be uniformly quasiconvex with modulus $\phi$ if, for all
$x,y\in\dom{\gh}$ and all $\lambda\in[0,1]$, it holds that
\[
 \gh(\lambda x+(1-\lambda)y) + \lambda(1-\lambda)\phi(\Vert x-y\Vert) \leq \bs\max\{\gh(x), \gh(y)\}.
\]
The class of uniformly quasiconvex functions includes strongly convex,
uniformly convex, and strongly quasiconvex functions (see Definition~\ref{def:unifquasi}).

A prominent factor in the analysis of optimization schemes for solving problems of the form \eqref{eq:mainproblem} is their convergence rate, particularly for problems involving a large number of variables. For PPA, i.e., scheme \eqref{eq:hippa} with $p=2$, a sublinear rate of convergence can typically be expected \cite{Gu2020Tight,Guler9291Convergence,Luque19849Asymptotic}, unless additional assumptions are imposed on the problem class. For example, it has been shown in \cite{rockafellar1976monotone} that under strong convexity of the objective function, the sequence generated by PPA converges linearly, and even superlinearly as 
$\gamma_k\to \infty$; see \cite{rockafellar1976monotone} for further details. Extending such results to broader classes of functions, such as uniformly convex or quasiconvex functions, poses significant challenges. For strongly quasiconvex functions \cite{GLM-survey,Lara2022strongly,polyak1966existence}, it has been shown in \cite[Theorem 10]{Lara2022strongly} that PPA generates a sequence that converges to a global minimizer. This result has inspired further research into the convergence behavior of proximal-point methods for nonconvex equilibrium problems, including investigations on linear convergence; see, for instance, \cite{Iusem2022proximal, Iusem20204two-step}. However, the class of strongly quasiconvex functions does not cover many important applications, such as those involving uniformly convex problems. This gap motivates the study of broader classes of nonconvex problems that include uniformly convex functions.

Recent theoretical and empirical advances have investigated the fast convergence rates and complexity of high-order proximal-point methods in both convex and nonconvex settings \cite{Ahookhosh23, Ahookhosh24, Kabgani24itsopt,Kabgani25First,Kabgani25itsdeal,Kabganidiff,Kabgani25fundamental,nesterov2021inexact,Nesterov2023a}. In particular, the fundamental properties of the high-order proximal operator and the Moreau envelope have been studied in \cite{Kabganidiff,Kabgani24itsopt}. In \cite{Kabgani24itsopt}, 
an inexact two-level smoothing optimization framework (ItsOPT) was proposed in which high-order proximal subproblems are solved inexactly at the lower level to construct an inexact oracle for the Moreau envelope. Then, a zero-, first-, or second-order method is developed at the upper level using this oracle, where a boosted HiPPA is introduced that converges to a \textit{proximal fixed point} (see Definition~\ref{def:critic}), which may correspond to either a local minimizer or a saddle point. Identifying a specific class of functions for which HiPPA or its variants are guaranteed to converge to a global minimizer with fast convergence rates remains an open and challenging problem, which we explore in the remainder of this paper.

\subsection{{\bf Contribution}}\label{sec:contribution}
Our contribution is twofold:

\begin{enumerate}
    \item[{\bf (i)}] {\bf Fundamental properties of uniformly quasiconvex functions.} We characterize the class of uniformly quasiconvex functions, study its calculus (addition, scalar multiplication, quotient, and composite functions), and explore its coercivity.
    Crucially, we demonstrate that this class inherently excludes saddle points by proving the equivalence between their stationary points (Definition~\ref{def:stationary}) and their unique global minimizers, thus demons\-tra\-ting the absence of spurious local minima or saddle points (Theorem~\ref{th:unifisdir}).

    \item[{\bf (ii)}] 
    {\bf Convergence and complexity analysis of HiPPA.} We provide a simple convergence and complexity analysis to a global minimizer of \eqref{eq:mainproblem} for HiPPA thanks to the desirable properties of the Moreau envelope. Our analysis reveals the following convergence behaviors of HiPPA when the uniform quasiconvexity modulus $\phi$ admits a power function of degree $q$ as a lower bound, i.e., $\phi(t) \geq c t^q$ for some $c>0$, on an interval $\mathcal{I}$: (i) for $q\in (1,2)$  and $\mathcal{I}=[0,1)$, HiPPA exhibits a local linear rate for $p\in [q,2)$;  (ii) HiPPA converges linearly when $p=2$, $q=2$, and also when $p=q>2$, provided that $\mathcal{I}=[0,\infty)$; (iii) for $q\geq 2$ and $\mathcal{I}=[0,\infty)$, HiPPA achieves a superlinear rate for $p>q$. To the best of our knowledge, this is the first comprehensive study of convergence rates of HiPPA. Surprisingly, some of these results are novel even for strongly or uniformly convex functions.
\end{enumerate}

\subsection{{\bf Organization}}\label{sec:Organization}
The remainder of this paper is structured as follows. In Section~\ref{sec:preliminaries}, we provide essential preliminaries and notation. In Section~\ref{sec:ben}, we study the fundamental properties of the class of uniformly quasiconvex functions. In Section~\ref{sec:hoppaexact}, we investigate the asymptotic convergence of HiPPA with a particular focus on the convergence rate. 
In Section~\ref{sec:conclusion}, we deliver our concluding remarks.


\section{Preliminaries and notation} \label{sec:preliminaries}

This section introduces the foundational notation, definitions, and tools necessary for our study of the high-order proximal-point algorithm for uniformly quasiconvex functions.

Throughout this paper, we denote by $\R^n$ the $n$-dimensional \textit{Euclidean space}, equipped with 
the \textit{Euclidean norm} $\Vert\cdot\Vert=\sqrt{\langle\cdot, \cdot\rangle}$, where
$\langle\cdot, \cdot\rangle$ is the standard \textit{inner product}. 
The set of \textit{natural numbers} including zero is denoted by $\Nz := \mathbb{N}\cup\{0\}$.
For $\ov{x}\in \R^n$ and $r>0$, the \textit{open ball} centered at $\ov{x}$ with radius $r$
is denoted by $\mb(\ov{x}; r)$.
Given a set $C\subseteq\R^n$, we denote its \textit{interior}, \textit{closure}, 
and \textit{convex hull} by $\inter{C}$, $\bs{\rm cl}(C)$, and $\bs{\rm conv}(C)$, respectively.
The \textit{indicator function} of $C$ is defined as $\iota_C (x)=0$ if $x \in C$, and 
$\iota_{C} (x) = +\infty$ if $x \not\in C$. 
For a convex set $C\subseteq \mathbb{R}^n$ and $\ov{x}\in C$, the cone of \textit{feasible directions} at
$\ov{x}$ is
    \[
    D_C(\ov{x}) := \left\{ d \in \R^{n}\mid ~ \text{there exists} ~ \delta
     > 0~\text{such that} ~ \ov{x} + \lambda d \in
     C ~ \text{for all}~ \lambda \in \, (0, \delta) \right\}.
    \]
If $C=\R^n$, then $D_C(\ov{x})=\R^n$. 
The function $\sigma_C(x) := \bs\sup_{k \in C} \langle k , x\rangle$ denotes the \textit{support function} of the set
$C$.

For $p > 1$, the gradient of $x\mapsto\frac{1}{p}\Vert x \Vert^p$ is given by
\[
\nabla\left(\frac{1}{p}\Vert x \Vert^p\right) = 
\begin{cases} 
\Vert x \Vert^{p-2} x & \text{if } x \neq 0, \\
0 & \text{if } x = 0.
\end{cases}
\]
For $p \in (1, 2)$, we write
$\nabla\left(\frac{1}{p}\Vert x \Vert^p\right) = \Vert x \Vert^{p-2} x$, for all $x\in\R^n$,
with the convention that $\frac{0}{0} = 0$ when $x = 0$.
Let $p>1$ and $x, y\in \R^n$ be arbitrary.
Since the function $\frac{1}{p}\Vert \cdot\Vert^p$ is convex, we arrive at
\[
  \frac{1}{p}\Vert x\Vert^p -\frac{1}{p}\Vert x-y\Vert^p\geq  \Vert x-y\Vert^{p-2}\langle x-y, x-(x-y)\rangle =\Vert x-y\Vert^{p-2}\langle x-y,y\rangle,
\]
leading to
  \begin{equation}\label{eq:ineqp}
   \Vert x-y\Vert^p\leq \Vert x\Vert^p-p\Vert x-y\Vert^{p-2}\langle x-y, y\rangle.
  \end{equation}
Moreover, it holds that
    \begin{equation}\label{eq:eqp=21}
   \langle x - z, y - x \rangle = \frac{1}{2} \Vert z - y \Vert^{2} - \frac{1}{2} \Vert x - z \Vert ^{2} - \frac{1}{2} \Vert y - x \Vert ^{2}.
  \end{equation}

The \textit{effective domain} of
a function $\gh: \R^n \to \Rinf := \R \cup \{+\infty\}$ is given by
        $\dom{\gh}:= \{x \in \R^n \mid \gh(x) < +\infty\}$,
and $\gh$ is called \textit{proper} if $\dom{\gh}\neq\emptyset$.
The \textit{sublevel set} of $\gh$ at height $\eta \in \R$ is
    $
    \mathcal{L}(\gh, \eta):=\{x \in \R^n \mid \gh(x) \leq \eta\}.
   $
The set of \textit{global minimizers} of $\gh$ over $C\subseteq \mathbb{R}^n$ is denoted by
    $\argmin{x \in C} \gh(x)$.
A function $\gh$ is \textit{lower semicontinuous} (lsc) at $\ov{x} \in \R^n$ if, for any sequence 
$\{x^k\}_{k\in \mathbb{N}} \subseteq \R^n$ with $x^k \to \ov{x}$, it holds that
$\bs\liminf_{k \to +\infty} \gh(x^k) \geq \gh(\ov{x})$.
For a sequence $\{x^k\}_{k\in \Nz}$, a point $\ov{x}\in\R^n$ is a \textit{limit point} 
if $x^k \to \ov{x}$, and $\widehat{x}\in\R^n$ is a \textit{cluster point} if there exists an infinite subset $J \subseteq \Nz$ such that the subsequence $\{x^j\}_{j \in J}$ satisfies
$x^j \to \widehat{x}$. The set of all cluster points of $\{x^k\}_{k \in \Nz}$ is denoted by
 $\Omega(x^k)$.

A proper function $\gh: \R^n \to \Rinf$ is \textit{Fr\'{e}chet differentiable} at 
$\ov{x} \in \inter{\dom{\gh}}$ with \textit{Fr\'{e}chet derivative} 
$\nabla \gh(\ov{x})$ if 
\[
\mathop{\bs\lim}\limits_{x \to \ov{x}} \frac{\gh(x) - \gh(\ov{x}) - \langle \nabla \gh(\ov{x}), x - \ov{x} \rangle}{\Vert x - \ov{x}\Vert} = 0.
\]
The Dini lower directional derivative of 
$\gh: \R^n \to \Rinf$ at $\ov{x} \in \dom{\gh}$ in direction $d \in \R^n$ is
\[
 \gh^{D-} (\ov{x}; d) := \mathop{\bs\liminf}\limits_{t \downarrow 0} \frac{\gh(\ov{x} + td) - \gh(\ov{x})}{t}. 
\]
If $\gh, \vartheta: \R^n \to \R$ are real-valued and $\vartheta$ is Fr\'{e}chet differentiable at $\ov{x}$, then 
\begin{equation}\label{eq:sumruldini}
 (\gh+\vartheta)^{D-} (\ov{x}; d) = \gh^{D-} (\ov{x}; d) + \langle \nabla \vartheta(\ov{x}), d \rangle,
 \qquad \forall d\in \R^n.
\end{equation}


\section{Uniform quasiconvexity: Characterization and properties}\label{sec:ben}

This section explores the properties of uniformly quasiconvex functions, providing characterizations and analytical tools essential for analyzing the convergence of HiPPA in Section~\ref{sec:hoppaexact}. 
We begin with a formal definition of uniform quasiconvexity and related convexity notions, followed by preliminary calculus rules. We then present a characterization linking uniform quasiconvexity to its behavior along line segments, conditions for supercoercivity, and properties of stationary points.

We first recall various notions of convexity and quasiconvexity.
\begin{definition}[Convexity and quasiconvexity]\label{def:unifquasi}
Let $\gh: \R^n\to \Rinf$ be a proper function with a convex domain, and let
$\phi: [0, +\infty) \to [0, +\infty]$ be a modulus, that is, an increasing function that vanishes only at $0$. The function $\gh$ is called:
\begin{enumerate}[label=(\textbf{\alph*}), font=\normalfont\bfseries, leftmargin=0.7cm]
 \item \textit{convex} if, for all $x, y \in \dom{\gh}$ and $\lambda \in [0, 1]$,
 \begin{equation}\label{def:convex}
  \gh(\lambda x + (1-\lambda)y) \leq \lambda \gh(x) + (1 - \lambda) \gh(y);
 \end{equation}
 
  \item \textit{uniformly convex} with modulus $\phi$ if, for all $x, y \in \dom{\gh}$ and $\lambda \in [0, 1]$, 
 \begin{equation}\label{eq:uniformlyconv}
 \gh(\lambda x+(1-\lambda)y) + \lambda(1-\lambda)\phi(\Vert x-y\Vert) \leq \lambda \gh(x) + (1-\lambda) \gh(y);
 \end{equation}

 \item  \textit{quasiconvex} if, for all $x, y \in \dom{\gh}$ and $\lambda\in [0,1]$,
 \begin{equation}\label{def:qcx}
  \gh(\lambda x + (1-\lambda) y) \leq \bs\max \{\gh(x), \gh(y)\};
 \end{equation}

  \item \textit{semistrictly quasiconvex} if, for all $x, y \in \dom{\gh}$ 
 with $\gh(x) \neq \gh(y)$ and $\lambda \in (0, 1)$,
 \begin{equation}
  \gh(\lambda x + (1-\lambda)y) < \bs\max \{\gh(x), \gh(y)\};
 \end{equation}

 \item \textit{uniformly quasiconvex} with modulus $\phi$ if, for all $x, y \in \dom{\gh}$ and $\lambda \in [0, 1]$,
\begin{equation}\label{eq:uniformlyquasi}
 \gh(\lambda x+(1-\lambda)y) + \lambda(1-\lambda)\phi(\Vert x-y\Vert) \leq \bs\max\{\gh(x), \gh(y)\}.
\end{equation}
\end{enumerate}
A function $\gh$ is \textit{strictly convex} (respectively, \textit{strictly quasiconvex}) 
if the inequality in \eqref{def:convex} (resp. \eqref{def:qcx}) is strict for $x\neq y$ and $\lambda\in (0,1)$. If $\phi(t)=\frac{\rho}{2}t^2$ for some $\rho > 0$, a uniformly convex (respectively, uniformly quasiconvex) function is termed \textit{strongly convex} (respectively, \textit{strongly quasiconvex}).
The relationships among these notions are illustrated in Figure~\ref{fig:quasicrel}, where quasiconvex is abbreviated as ``qcx".
\end{definition}

\begin{figure}[h]
    \centering
    \scalebox{0.8}{
    \begin{tikzpicture}[
        node distance=1cm and 0.8cm,
        >=Stealth,
        every node/.style={align=center, font=\small},
        convexbox/.style={rectangle, rounded corners, draw=black, fill=gray!10, thick, minimum width=2.5cm, minimum height=0.8cm},
        qcxbox/.style={rectangle, rounded corners, draw=blue!50, fill=blue!5, thick, minimum width=2.5cm, minimum height=0.8cm},
        arrow/.style={->, thick},
        dashedarrow/.style={->, thick, dashed}
    ]
        \node (scx) [convexbox] {strongly convex};
        \node (ucx) [convexbox, right=of scx] {uniformly convex};
        \node (stcx) [convexbox, right=of ucx] {strictly convex};
        \node (cx) [convexbox, right=of stcx] {\textbf{convex}};

        \node (sqcx) [qcxbox, below=of scx] {strongly qcx};
        \node (uqcx) [qcxbox, below=of ucx] {uniformly qcx};
        \node (stqcx) [qcxbox, below=of stcx] {strictly qcx};
        \node (ssqcx) [qcxbox, below=of cx] {semistrictly qcx};
        \node (qcx) [qcxbox, right=of ssqcx] {\textbf{qcx}};

        \draw[arrow] (scx) -- (ucx);
        \draw[arrow] (ucx) -- (stcx);
        \draw[arrow] (stcx) -- (cx);
        \draw[arrow] (cx) -- (qcx);
        
        \draw[arrow] (sqcx) -- (uqcx);
        \draw[arrow] (uqcx) -- (stqcx);
        \draw[arrow] (stqcx) -- (ssqcx);
        \draw[arrow] (ssqcx) -- node[above, pos=0.4] {lsc} (qcx);

        \draw[arrow] (scx) -- (sqcx);
        \draw[arrow] (ucx) -- (uqcx);
        \draw[arrow] (stcx) -- (stqcx);
        \draw[arrow] (cx) -- (ssqcx);
    \end{tikzpicture}}
    \caption{Relationships among notions introduced in Definition~\ref{def:unifquasi}.}
    \label{fig:quasicrel}
\end{figure}
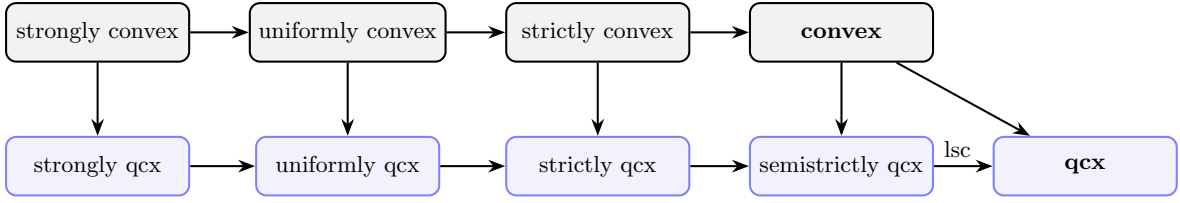


A highly significant illustration of a uniformly quasiconvex function is the function
$x\mapsto\Vert x\Vert^q$ for $q>0$, analyzed in the following example.

\begin{example}\label{ex:norm:unif}
Let $\gh: \mathbb{R}^{n} \rightarrow \mathbb{R}$ be a function given by
$\gh(x)=\Vert x\Vert^{q}$ for $q>0$. We examine its uniform quasiconvexity properties across different values of $q$.
 \begin{itemize}
  \item If $0 < q \leq 1$, then $\gh$ is strongly quasiconvex, and thus uniformly quasiconvex, on the open ball $\mb(0; r)$ for any $r>0$, as established in
\cite[Theorem 2]{jovanovivc1996note}, \cite[Corollary 3.9]{nam2024strong}, and \cite[page 27]{vladimirov1978uniformly}.

 \item  If $q\in (1, 2)$, it is shown in \cite[Lemma~2]{Kabganidiff} that, for any
 $r>0$  and $a, b\in \mb(0; r)$,
\[
\langle \Vert a\Vert^{q-2}a - \Vert b\Vert^{q-2}b, a-b\rangle\geq \kappa(q) r^{q-2}\Vert a - b\Vert^2,
\]
where
\[
\kappa(t):=\left\{
   \begin{array}{ll}
     \frac{(2+\sqrt{3})(t-1)}{16} & t\in (1, \widehat{t}], \\[0.2cm]
      \frac{2+\sqrt{3}}{16}\left(1-\left(3-\sqrt{3}\right)^{1-t}\right) ~~& t\in [\widehat{t},2),
   \end{array}\right.
\]
and $\widehat{t}$ is the solution to $\frac{t(t-1)}{2} = 1 - \left[1 + \frac{(2-\sqrt{3})t}{t-1}\right]^{1-t}$ on $(1, 2]$, and is determined numerically as $\widehat{t} \approx 1.3214$. For simplicity, we write $\kappa(t)$ as $\kappa_t$.
Hence, $\gh(x)=\lVert x\rVert^{q}$ is strongly convex on $\mb(0; r)$ 
with modulus $\phi(t)=\frac{\kappa_q r^{q-2}}{2}t^2$
\cite[Lemma~4.2.1]{Nesterov2018}.

\item If $q=2$, $\gh$ is strongly convex on $\R^n$ with modulus $\phi(t) = t^2$ \cite[Example~10.9]{Bauschke17}. Specifically, for all $x, y \in \R^n$ and $\lambda \in [0, 1]$,
\begin{equation}\label{eq:eqp=22}
\Vert \lambda x + (1-\lambda) y \Vert ^{2} \leq \lambda\Vert x \Vert ^{2} + (1 -\lambda ) \Vert y\Vert ^{2} - \lambda (1 -\lambda ) \Vert x - y \Vert ^{2}.
\end{equation}

\item If $q>2$, $\gh$ is uniformly convex on $\R^n$ \cite[Example~10.16]{Bauschke17}.
Assume that $\widehat\phi$ is the modulus of this function. Since $\frac{q}{2}\geq 1$ and $\Vert\cdot\Vert^q=\left(\Vert\cdot\Vert^2\right)^{\frac{q}{2}}$, based on \cite[Proposition~10.15, Example 10.16]{Bauschke17},
\[
\widehat\phi(\cdot)\geq 2^{1-q} \bs\min\left\{\frac{q}{2}2^{1-\frac{q}{2}}, 1-2^{-\frac{q}{2}}\right\}\vert \cdot\vert^q.
\]
Note that
$\bs\min\left\{\frac{q}{2}2^{1-\frac{q}{2}}, 1-2^{-\frac{q}{2}}\right\}\geq 2^{-\frac{q}{2}}$.
Thus,
\[
2^{1-q} \bs\min\left\{\frac{q}{2}2^{1-\frac{q}{2}}, 1-2^{-\frac{q}{2}}\right\}\vert \cdot\vert^q\geq 2^{1-q}2^{-\frac{q}{2}}\vert \cdot\vert^q=2^{-\frac{3q-2}{2}}\vert \cdot\vert^q.
\]
Hence, for each $x, y\in \R^n$ and $\lambda\in (0,1)$, by setting $\widehat{\sigma}_q := \left(\frac{1}{2}\right)^\frac{3q-2}{2}$, we have
\begin{equation}\label{eq:unifcon:normp}
\Vert \lambda x+(1-\lambda) y\Vert^q\leq \lambda \Vert x\Vert^q+(1-\lambda)\Vert y\Vert^q - \lambda(1-\lambda)\widehat{\sigma}_q\Vert x-y\Vert^q.
\end{equation}
   \end{itemize}
From the relationships in Figure~\ref{fig:quasicrel}, $\gh(x) = \Vert x \Vert^q$ is uniformly quasiconvex on $\mb(0; r)$ for $q \in (0, 2)$ and on $\R^n$ for $q \geq 2$.
\end{example}

\begin{remark}
In Example~\ref{ex:norm:unif}, we established the uniform quasiconvexity of the function $\gh(x) = \Vert x \Vert^q$ for all $q>0$, either locally on bounded balls or globally on $\R^n$, depending on the value of $q$. 
 This result does not extend to the $q$-norm functions $\gh_1(x) = \Vert x \Vert_q^q = \sum_{i=1}^n \vert x_i \vert^q$ or $\gh_2(x) = \Vert x \Vert_q = \left( \sum_{i=1}^n \vert x_i \vert^q \right)^{1/q}$ for $q\in (0,1)$ in $\R^n$ with $n \geq 2$, as neither is quasiconvex.
For instance, consider $x = (1, 0), y = (0, 1)$ in $\R^{2}$ and $z = \frac{x+y}{2} = (\frac{1}{2}, \frac{1}{2})$. 
Then $\bs\max\{\gh_i(x), \gh_i(y)\} =1$, for $i=1, 2$, but 
\[
\gh_1 (z) = 2 \frac{1}{2^{q}} = 2^{1-q} > 1,~~\gh_2(z) = (2^{1-q})^{\frac{1}{q}} = 2^{\frac{1}{q} - 1} > 1, \qquad \forall ~ q \in \, (0, 1).
\]
These inequalities imply that
$\gh_i(z) > \bs\max\{\gh_i(x), \gh_i(y)\}$, and hence $\gh_i$ is not quasiconvex for $i=1, 2$.
\end{remark}

In general, the sum of uniformly quasiconvex functions is not necessarily quasiconvex (see \cite{flores2021characterizing,zaffaroni2004every}). In \cite[Example~6.5]{nam2024strong}, it is shown that while the functions 
$x\mapsto \gf(x-1)$ and $x\mapsto \gf(x+1)$ are uniformly quasiconvex, where
\begin{equation}\label{eq:funphinotu}
  \gf(x):=\begin{cases} 
x^2 & \text{if } x \neq 0, \\
-1 & \text{if } x = 0,
\end{cases}  
\end{equation}
the function $\gh(x)= \gf(x-1)+\gf(x+1)$ fails to be quasiconvex.

Next, we present results on the calculus of uniform quasiconvexity under addition and scalar multiplication.

\begin{proposition}[Addition and scalar multiplication]\label{pro:cal:sc}
Let $\gh:\R^n\to \Rinf$ be a proper and uniformly quasiconvex function with modulus $\phi$. Let
$C\subseteq\R^n$ be a nonempty convex set, $a\in \R$, and $b>0$. Then, the following statements hold:
\begin{enumerate}[label=(\textbf{\alph*}), font=\normalfont\bfseries, leftmargin=0.7cm]
\item\label{pro:cal:sc:a} $\gh+a$ is a uniformly quasiconvex function with modulus $\phi$;
\item\label{pro:cal:sc:b} $b\gh$ is a uniformly quasiconvex function with modulus $b\phi$;
\item\label{pro:cal:sc:c}  $\gh+\iota_C$ is a uniformly quasiconvex function with modulus 
$\phi$.
\end{enumerate}
\end{proposition}

The following proposition, which generalizes \cite[Proposition~4.1]{Iusem20204two-step}, introduces a\-ddi\-tio\-nal members of the class of uniformly quasiconvex functions, which is particularly important in economics \cite{cambini2008generalized,iusem2018second,Iusem20204two-step}.
The proof follows directly from the reasoning in the cited proposition and is omitted for brevity.

\begin{proposition}[Uniform quasiconvexity of quotient functions]\label{prop:quotient}
Let $h, g: \R^n \to \R$ and define $\gh(x):=\frac{h(x)}{g(x)}$, where $g(x)\neq 0$ for each $x\in \R^n$. Suppose $h$ is uniformly convex with modulus $\phi$, and $g$ is finite, positive, and bounded above by $M$. If any of the following conditions hold:
\begin{enumerate}[label=(\textbf{\alph*}), font=\normalfont\bfseries, leftmargin=0.7cm]
\item $g$ is affine,
\item $h$ is nonnegative and $g$ is concave,
\item $h$ is nonpositive and $g$ is convex,
\end{enumerate}
then $\gh$ is uniformly quasiconvex with modulus $\frac{\phi}{M}$.
\end{proposition}

We introduce key definitions and concepts used in the subsequent analysis to provide a sufficient condition for uniform quasiconvexity of composite functions.
We begin with the notion of a convex-pair, which facilitates the analysis of composite functions.

\begin{definition}[Convex-pair]\label{def:semipair}
  Let $\gf: \R^n\to\R^m$ be a mapping and $\gh:\R^m\to \Rinf$ be a proper function. 
  Let $C\subseteq \R^n$ be a convex set such that $\gf(C)\subseteq \dom{\gh}$.
  The pair $(\gf, \gh)$ is called a \textit{convex-pair} with respect to $C$ if, for all $x,y\in C$ and 
  $\lambda\in [0,1]$, 
  \[
  \gh(\gf(\lambda x+(1-\lambda)y)) \leq \gh(\lambda \gf(x)+(1-\lambda)\gf(y)).
  \]
\end{definition}

To clarify the convex-pair notion, we provide an example showing how common
function classes satisfy this property. This example highlights the interplay
between the convexity (or concavity) of $\gf$ and the monotonicity of $\gh$.
\begin{example}
If $\gf: \R^n\to\R$ is a convex (respectively, concave) function on a convex set $C$, and $\gh:\R\to \Rinf$ is nondecreasing (respectively, nonincreasing) on $\gf(C)$, then 
$(\gf, \gh)$ is a convex-pair with respect to $C$. 
\end{example}

Next, we recall affine mappings on convex sets, which are instrumental in cons\-truc\-ting composite functions that preserve quasiconvexity.
\begin{definition}[Affine mapping]\label{def:simline}
An operator $T:\R^n\to\R^m$ is called an \textit{affine mapping} on a convex set $C\subseteq \R^n$ if, for all $\lambda\in [0,1]$ and $x, y\in C$, 
\[
T(\lambda x+(1-\lambda) y)=\lambda T(x)+(1-\lambda) T(y).
\]
\end{definition}

\begin{remark}\label{rem:homconp}
\begin{enumerate}[label=(\textbf{\alph*}), font=\normalfont\bfseries, leftmargin=0.7cm]
  \item\label{rem:homconp:a}  Linear operators are affine mappings on $\R^n$, but not vice versa. For example, for each matrix $A\in \R^{m\times n}$ and non-zero vector $b\in\R^m$, the operator $S(x)=Ax-b$ is an affine mapping but is not a linear operator.
  \item \label{rem:homconp:b} If $T:\R^n\to\R^m$ is an affine mapping on a convex set $C\subseteq \R^n$, and $\gf: \R^m\to\Rinf$ is a proper function such that $T(C)\subseteq \dom{\gf}$, then $(T, \gf)$ is a convex-pair with respect to $C$.
  \end{enumerate}
\end{remark}

In the following, we introduce the notion of uniform regularity, which ensures that operators preserve distances appropriately for uniform quasiconvexity.

\begin{definition}[Uniform regularity]\label{def:unifreg}
An operator $T:\R^n\to\R^m$ is said to be \textit{uniformly regular} on a convex set $C\subseteq \R^n$ with modulus $\phi$ if
\[
\phi(\Vert x - y\Vert)\leq \Vert T(x) - T(y)\Vert, \qquad \forall x, y\in C.
\]
\end{definition}


\begin{remark}\label{rem:unifnorm}
\begin{enumerate}[label=(\textbf{\alph*}), font=\normalfont\bfseries, leftmargin=0.7cm]
\item\label{rem:unifnorm:a} We recall that an operator $T:\R^n\to\R^n$ is called \textit{uniformly monotone} on a convex set $C$ if there exists a modulus $\phi$ such that, for all $x, y\in C$, 
\[
\langle x - y, T(x) - T(y)\rangle\geq \phi(\Vert x - y\Vert).
\]
By the Cauchy–Schwarz inequality, uniform monotonicity implies
\begin{equation}\label{eq:inLip}
\ov{\phi}(\Vert x - y\Vert):=\frac{\phi(\Vert x - y\Vert)}{\Vert x - y\Vert}\leq \Vert T(x) - T(y)\Vert,\qquad \forall x,y\in C, x\neq y.
\end{equation}
For $\phi(t)=\rho t^2$, with $\rho>0$, the inequality \eqref{eq:inLip} is known as the \textit{reverse Lipschitz} or \textit{inverse Lipschitz} condition, and it implies uniform regularity. Additionally, for single-valued and bijective operators $T$, this condition is referred to as \textit{metric regularity}. 
In particular, the class of uniformly monotone operators whose modulus $\vartheta$ can be written as $\vartheta(t)=t\phi(t)$, where $\phi$ is a modulus, forms a subclass of uniformly regular operators.

\item\label{rem:unifnorm:b} For the operator $S(x)=Ax-b$, if all singular values are positive, or equivalently when $A$ is full column rank, then
\[
\Vert S(x)-S(y)\Vert=\Vert A(x-y)\Vert\geq \sigma_{\min}\Vert x-y\Vert,
\]
where $\sigma_{\min}>0$ denotes the smallest singular value of $A$
\cite[Appendix~A.5]{boyd2004convex}. In this case, $S$ is uniformly regular with modulus
$\phi(t)=\sigma_{\min}t$.
 
\end{enumerate}
\end{remark}

Here, we present a sufficient condition for the uniform quasiconvexity of composite functions. 
The simple proof is omitted.

\begin{theorem}[Uniform quasiconvexity of composite functions]\label{th:com:unif}
Let $\gf: \R^n\to\R^m$ and $\gh:\R^m\to \Rinf$ 
be a convex-pair with respect to a convex set $C\subseteq \R^n$.
If $\gh$ is uniformly quasiconvex with modulus $\phi$ and $\gf$ is uniformly regular with modulus $\vartheta$, then the composite function $x\mapsto \gh(\gf(x))$ is uniformly quasiconvex on $C$ with modulus $\phi\circ \vartheta$.
\end{theorem}

\begin{corollary}[Uniform quasiconvexity of norm–affine compositions]\label{cor:com:unif}
Let $T:\R^n\to\R^m$ be an affine mapping that is uniformly regular.
Then, the function
$ \gh(x)=\lVert T(x)\rVert^{q}$
is uniformly quasiconvex on the ball $\mb(0; r)$ for any $r>0$
and all $q > 0$.
\end{corollary}
\begin{proof}
It follows from Remark~\ref{rem:homconp}~\ref{rem:homconp:b}, Theorem~\ref{th:com:unif}, and Example~\ref{ex:norm:unif}.
\end{proof}

The next examples illustrate the potential for developing novel models capable of addressing practical, real-world problems.
\begin{example}\label{cor:com:unif:matrix}
Let $A\in \R^{m\times n}$ be full column rank and $b\in \R^{m}$. It follows from Remarks~\ref{rem:homconp}~\ref{rem:homconp:a}~and~\ref{rem:unifnorm}~\ref{rem:unifnorm:b}, Theorem~\ref{th:com:unif}, and Example~\ref{ex:norm:unif}
that the function $\gh(x)=\lVert Ax -b\rVert^{q}$
is uniformly quasiconvex on $\mb(0; r)$ with $r>0$
for all $q > 0$.
\end{example}

\begin{example}
Consider the operator $\gf(x)=A^\top(Ax-b)+\alpha x$, where $A\in\R^{m\times n}$, $b\in\R^m$, and $\alpha>0$. This operator coincides with the gradient of the Tikhonov-regularized least-squares functional
\[
f(x)=\frac{1}{2}\Vert Ax-b\Vert^2+\frac{\alpha}{2}\Vert x\Vert^2. 
\]
Since $A^\top A$ is symmetric positive semidefinite, it holds for all $v\in\R^n$ that
\[
\langle (A^\top A+\alpha I)v,v\rangle=\Vert Av\Vert^2+\alpha\Vert v\Vert ^2\geq \alpha\Vert v\Vert^2.
\]
Applying this inequality with $v=x-y$ and invoking the Cauchy--Schwarz inequality yields
\[
\alpha\Vert x-y\Vert^2 \leq \langle \gf(x)-\gf(y),x-y\rangle
\leq \Vert \gf(x)-\gf(y)\Vert \Vert x-y\Vert.
\]
For $x\neq y$, dividing by $\|x-y\|$ gives
$\Vert \gf(x)-\gf(y)\Vert \geq \alpha\Vert x-y\Vert$.
Hence, $\gf$ is uniformly regular on $\R^n$ (and therefore on any convex set $C\subseteq\R^n$) with linear modulus $\vartheta(t)=\alpha t$.
Let $\gh(u)=\|u\|^q$ with $q>0$. By Example~\ref{ex:norm:unif}, the function $\gh$ is uniformly quasiconvex on bounded sets. Consequently, by Theorem~\ref{th:com:unif}, the composite function
$\gh(\gf(x))=\lVert A^\top(Ax-b)+\alpha x\rVert^{q}$ is uniformly quasiconvex on the ball $\mb(0;r)$ for any $r>0$ and any $q>0$. This construction yields a residual-based reformulation of the problem $\bs\min_{x\in\R^n} f(x)$, namely, $\bs\min_{x\in\R^n} \gh(\gf(x))=\bs\min_{x\in\R^n}\Vert \nabla f(x)\Vert^q$, 
commonly referred to as \textit{gradient norm minimization}.
Despite its potential nonsmoothness, this objective enjoys favorable geometric properties such as the absence of spurious stationary points owing to its uniform quasiconvexity. These properties will be exploited in the subsequent analysis.
\end{example}

\subsection{Characterization and \texorpdfstring{$m$}{m}-supercoercivity}

We begin with the following theorem that generalizes \cite[Theorem~1]{jovanovivc1996note} by extending the line-segment characterization of strong quasiconvexity to uniform quasiconvexity with a modulus. The proof follows the same argument as \cite[Theorem~1]{jovanovivc1996note}.

\begin{theorem}[Line-segment characterization]
\label{thm:line-segment-uniform-quasiconvex}
Let $C \subseteq \R^n$ be a nonempty convex set. A proper function $\gf:\R^n\to\Rinf$ is uniformly quasiconvex on $C$ with modulus $\phi$ if and only if, for any distinct $x, y \in C$, the function
\[
\gh(t) = \gf\left(y + \frac{t}{\|x-y\|}(x-y)\right), \quad t \in \mathbb{R},
\]
is uniformly quasiconvex on $[0, \|x-y\|]$ with the same modulus $\phi$.
\end{theorem}

For continuously differentiable functions, we present a differential characterization of uniform quasiconvexity. 
\begin{theorem}[Differential characterization]\label{th:chardiff}
Let $\gf:\R^n\to\R$ be continuously differentiable. Then,
\begin{enumerate}[label=(\textbf{\alph*}), font=\normalfont\bfseries, leftmargin=0.7cm]
\item \label{th:chardiff:a} if $\gf$ is uniformly quasiconvex with modulus $\phi$, then
\begin{equation}\label{char:uqcx}
 \gf(x) \leq \gf(y) \, \Longrightarrow \, 
\langle \nabla \gf (y), x - y \rangle \leq -\phi(\lVert y - x \rVert);
\end{equation}

\item \label{th:chardiff:b} if \eqref{char:uqcx} holds, then $\gf$ is uniformly quasiconvex with modulus $\phi/2$;

\item \label{th:chardiff:c} 
if $\gf$ is uniformly quasiconvex with modulus $\phi$, then the following generalized monotonicity condition holds for all $x, y\in \R^n$:
\begin{equation}\label{gen:mon:phi}
 \langle \nabla \gf (y), x - y \rangle > -\phi(\lVert y - x \rVert) \, \Longrightarrow \, \langle \nabla \gf (x), y - x \rangle \leq -\phi(\lVert y - x \rVert).
\end{equation}
\end{enumerate}
\end{theorem}
\begin{proof}
\ref{th:chardiff:a} See \cite[Theorem 1]{vladimirov1978uniformly}.
\\
\ref{th:chardiff:b} See \cite[Theorem 6]{vladimirov1978uniformly}.
\\
\ref{th:chardiff:c} Let $x, y \in \mathbb{R}^{n}$ be such that $\langle \nabla \gf (y), x - y \rangle > -\phi(\lVert y - x \rVert)$. Then, \eqref{char:uqcx} implies $\gf(x) > \gf(y)$. Applying \eqref{char:uqcx} again, 
we get $ \langle \nabla \gf(x), y - x \rangle \leq -\phi(\Vert y - x \Vert)$.
\end{proof}


We now address the coercivity properties of uniformly quasiconvex functions, which are critical to ensuring the existence of minimizers, as established later in Theo\-rem~\ref{th:unifisdir}. In Theorem~\ref{th:unquaissup}, we provide a sufficient condition for
$m$-supercoercivity of uniformly quasiconvex functions.
We recall definitions related to coercivity for a proper function $\gh: \R^n \to \Rinf$; cf. \cite{Cambini2003coercivity}.
\begin{enumerate}[label=(\textbf{\alph*}), font=\normalfont\bfseries, leftmargin=0.7cm]
\item\label{coer:msuper} $\gh$ is \textit{$m$-supercoercive} for some $m\in \mathbb{N}$ if,
$\mathop{\bs\liminf}\limits_{\Vert x \Vert \to +\infty} \frac{\gh(x)}{\Vert x\Vert^m} >0$;
\item \label{coer:super} $\gh$ is \textit{supercoercive} if,
$\mathop{\bs\lim}\limits_{\Vert x \Vert \to +\infty} \frac{\gh(x)}{\Vert x\Vert}=+\infty$;

\item \label{coer:coer} $\gh$ is \textit{coercive} if,
$\mathop{\bs\lim}\limits_{\Vert x \Vert \to +\infty} \gh(x) = +\infty$.
\end{enumerate}
For $m \geq 2$, $m$-supercoercivity implies supercoercivity. Moreover, supercoercivity implies  
$1$-supercoercivity, which in turn implies coercivity.   
The function $\gh(x) = |x|^{\frac{3}{2}}$ is supercoercive, but not $m$-supercoercive for any $m \geq 2$.  
Furthermore, $\gh(x) = |x|$ is $1$-supercoercive but not supercoercive, while $\gh(x) = \sqrt{\vert x\vert}$ 
is coercive but not $1$-supercoercive. 
\begin{theorem}[$m$-supercoercivity]\label{th:unquaissup}
Let $\gh :\R^n\to \Rinf$ be a proper and uniformly quasiconvex function with modulus $\phi$, and suppose
$\bs{\rm int}(\dom{\gh})\neq \emptyset$. If there exists some $m\in \mathbb{N}$ with $m\geq 2$ such that, for each $y\in \R^n$, we have
\begin{equation}\label{eq:th:unquaissup}
\mathop{\bs\liminf}\limits_{\Vert x\Vert\to +\infty}\frac{\phi(\Vert x - y\Vert)}{\Vert x\Vert^{m}}>0,
\end{equation}
then $\gh$ is $m$-supercoercive.
\end{theorem}
\begin{proof}
Since $\gh$ is uniformly quasiconvex, it follows from \cite{crouzeix2005continuity} 
that $\gh$ is Fr\'{e}chet differentiable almost everywhere in $\inter{\dom{\gh}}$, ensuring the existence of a point $\ov{x} \in \inter{\dom{\gh}}$ where $\gh$ is continuous.
Consider a sequence $\{x^k\}_{k\in \mathbb{N}}\subseteq \R^n$ such that $\Vert x^k\Vert\to +\infty$. Since $\left\{\frac{x^k}{\Vert x^k\Vert}\right\}_{k\in \mathbb{N}}$ is bounded, it has a convergent subsequence. Without loss of generality, we assume that $\frac{x^k}{\Vert x^k\Vert}\to u$. 
Choose $\varepsilon \in (0, 1)$ small enough such that
$y:= \ov{x} -\varepsilon u \in \inter{\dom{\gh}}$.
Define
\[y^k:=\frac{1+\Vert x^k\Vert-\varepsilon}{1+\Vert x^k\Vert}y+\frac{\varepsilon}{1+\Vert x^k\Vert}x^k.\]
Then, $y^{k} \rightarrow y + \varepsilon u = \ov{x}$, and by uniform quasiconvexity of $\gh$,
\[
\gh(y^k)\leq \bs\max\{\gh(y), \gh(x^k)\}-\frac{\varepsilon\left(1+\Vert x^k\Vert-\varepsilon\right)}{\left(1+\Vert x^k\Vert\right)^2}\phi(\Vert x^k - y\Vert).
\]
Dividing by $\Vert x^k\Vert^{m-1}$ yields
\begin{align*}
\frac{\gh(y^k)}{\Vert x^k\Vert^{m-1}}&
\leq \bs\max\left\{\frac{\gh(y)}{\Vert x^k\Vert^{m-1}}, \frac{\gh(x^k)}{\Vert x^k\Vert^{m-1}}\right\}-\frac{\varepsilon\left(1+\Vert x^k\Vert-\varepsilon\right)}{\Vert x^k\Vert^{m-1}\left(1+\Vert x^k\Vert\right)^2}\phi(\Vert x^k - y\Vert)
\\&
= \bs\max\left\{\frac{\gh(y)}{\Vert x^k\Vert^{m-1}}, \frac{\gh(x^k)}{\Vert x^k\Vert^{m-1}}\right\}-\frac{\frac{\varepsilon\left(1+\Vert x^k\Vert-\varepsilon\right)}{\Vert x^k\Vert}}{\frac{\left(1+\Vert x^k\Vert\right)^2}{\Vert x^k\Vert^2}}\frac{\phi(\Vert x^k - y\Vert)}{\Vert x^k\Vert^{m}}
\\&
= \bs\max\left\{\frac{\gh(y)}{\Vert x^k\Vert^{m-1}}, \frac{\gh(x^k)}{\Vert x^k\Vert^{m-1}}\right\}-\frac{\varepsilon\left(1+\frac{1-\varepsilon}{\Vert x^k\Vert}\right)}{\left(1+\frac{1}{\Vert x^k\Vert}\right)^2}\frac{\phi(\Vert x^k - y\Vert)}{\Vert x^k\Vert^{m}}.
\end{align*}
The continuity of $\gh$ at $\ov{x}$ ensures
$\gh(\ov{x})=\bs\liminf_{k\to\infty} \gh(y^k)$,
i.e.,
\[
\mathop{\bs\liminf}\limits_{k\to+\infty}\left(\bs\max\left\{\frac{\gh(y)}{\Vert x^k\Vert^{m-1}}, \frac{\gh(x^k)}{\Vert x^k\Vert^{m-1}}\right\}\right)\geq\varepsilon~\mathop{\bs\liminf}\limits_{k\to +\infty}\frac{\phi(\Vert x^k - y\Vert)}{\Vert x^k\Vert^{m}}>0.
\]
Since $\bs\lim_{k\to\infty}\frac{\gh(y)}{\Vert x^k\Vert^{m-1}}=0$, it follows that $\bs\liminf_{k\to+\infty}\frac{\gh(x^k)}{\Vert x^k\Vert^{m-1}}>0$.
This implies that
$\gh$ is $(m-1)$ -supercoercive and, hence, coercive.
Next, consider the sequence $z^k = \frac{1}{2}x^1+\frac{1}{2}x^k$. Since $\Vert z^k\Vert\to+\infty$, the coercivity of $\gh$ implies there exists a finite constant $c\in \R$ such that, for sufficiently large $k$,
\[
c\leq \gh(z^k)\leq \bs\max\{\gh(x^1), \gh(x^k)\} - \frac{1}{4}\phi(\Vert x^k- x^1\Vert)\leq 
\gh(x^k)- \frac{1}{4}\phi(\Vert x^k- x^1\Vert).
\]
Rearranging this inequality yields
\[
\frac{c}{\Vert x^k\Vert^m}+\frac{1}{4}\frac{\phi(\Vert x^k- x^1\Vert)}{\Vert x^k\Vert^m}\leq \frac{\gh(x^k)}{\Vert x^k\Vert^m}.
\]
Taking the limit as $k\to \infty$, we conclude that $\gh$ is $m$-supercoercive.
\end{proof}

\begin{remark}\label{rem:supercoercive}
We clarify the scope and limitations of Theorem~\ref{th:unquaissup} with the following observations:
\begin{enumerate}[label=(\textbf{\alph*}), font=\normalfont\bfseries, leftmargin=0.7cm]
\item \label{rem:supercoercive:a} 
If the modulus is $\phi(t)=\frac{1}{2}t^{2}$ and $m=2$, condition \eqref{eq:th:unquaissup} holds automatically.
In this case, Theorem~\ref{th:unquaissup} reduces to 
\cite[Theorem 1]{Lara2022strongly};

 \item Theorem~\ref{th:unquaissup} requires uniform quasiconvexity on the entire domain $\dom{\gh}$.
 If $\gh$ is uniformly quasiconvex only on a strict subset of the domain, 
 the result may not hold. 
 For example, consider $\gh(x)=\Vert x\Vert^\alpha$ with $\alpha\in (0,2)$, which is uniformly quasiconvex on every bounded ball but not $2$-supercoercive.
\end{enumerate}
\end{remark}

\subsection{No spurious local minima and saddle points}
Nonsmooth and nonconvex optimization problems often suffer from spurious local minima and saddle points that can trap iterative methods, hindering convergence to global solutions.
Uniformly quasiconvex functions offer a powerful remedy, free from such obstacles.
This property is transformative for high-order proximal-point methods, enabling them to efficiently and reliably locate global minimizers, even in complex settings.
This section explores this property, establishing that any stationary point is a global minimizer and is unique under practical assumptions. These insights lay a robust theoretical foundation for the algorithm developed in Section~\ref{sec:hoppaexact}, promising significant theoretical and practical advantages.

Let us begin by defining stationary points for constrained optimization problems, using the lower Dini directional derivative to accommodate nonsmooth functions.
\begin{definition}[Stationary point]\label{def:stationary}
Let $\gh: \R^n\to \Rinf$ be a proper function and $C\subseteq \R^n$ a closed convex set.
A point $\ov{x}\in \dom{\gh}\cap C$ is a \textit{stationary point} of $\gh$ on $C$ if, for all
$d \in D_C(\ov{x})$,
\[
    \gh^{D-} (\ov{x}; d) \geq 0.
\]
The set of stationary points is denoted by $\bs{\rm Dcrit}(\gh, C)$, or $\bs{\rm Dcrit}(\gh)$ if $C=\R^n$.
\end{definition}

Note that the stationarity condition in Definition~\ref{def:stationary} is closely related to the cla\-ssi\-cal 
necessary optimality conditions
in nonsmooth optimization
under local Lipschitz continuity, while remaining broadly applicable without this assumption.
\begin{remark}[Stationarity and necessary optimality conditions]\label{rem:onsta}
 The definition provided for stationary points in Definition~\ref{def:stationary} is motivated by necessary optimality conditions for problem~\eqref{eq:mainproblem}. 
 Assume $\gh$ is locally Lipschitz at $\ov{x}\in \dom{\gh}\cap C$.
Then, for all $d\in \R^n$,
$\gh^{D-} (\ov{x}; d)\leq \gh^{\circ} (\ov{x}; d)$ (see \cite[page~631]{WU2003Equivalence}), where
\[
\gh^{\circ} (\ov{x}; d) := {\mathop {\mathop {\bs\limsup }\limits _{y \rightarrow \ov{x},}} \limits _{t \downarrow 0}}\frac{\gh(y+td) - \gh(y)}{t},
\]
is the \textit{Clarke directional derivative} of $\gh$ at $\ov{x}$ in the direction $d$ \cite{clarke2013functional}.
The \textit{Clarke subdifferential} of $\gh$ at $\ov{x}$ is
$
\partial^{CL}\gh(\ov{x}):=\{\zeta\in\R^n\mid \gh^{\circ} (\ov{x}; d)\geq \langle\zeta, d\rangle,~\forall d\in \R^n\},
$
with $\gh^{\circ} (\ov{x}; d)=\bs\max\{\langle\zeta , d\rangle\mid \zeta\in \partial^{CL}\gh(\ov{x})\}$
\cite[Definition~10.3]{clarke2013functional}, i.e.,
$d\mapsto\gh^{\circ} (\ov{x}; d)$ is the support function of $\partial^{CL}\gh(\ov{x})$.
If $\ov{x}$ is stationary, then $\gh^{\circ}(\ov{x}; d) \geq 0$ for all $d \in D_C(\ov{x})$, implying
$\bs\max\{\langle\zeta , d\rangle\mid \zeta\in \partial^{CL}\gh(\ov{x})\}\geq 0$ for all
$d \in \bs{{\rm cl}}(D_C(\ov{x}))$.
This leads to
\[
\sigma_{\partial^{CL}\gh(\ov{x})}(d)+\iota_{\bs{{\rm cl}}(D_C(\ov{x}))}(d)=\mathop{\bs\max}_{\zeta\in \partial^{CL}\gh(\ov{x})}\langle\zeta , d\rangle+\iota_{\bs{{\rm cl}}(D_C(\ov{x}))}(d)\geq 0,
\qquad\forall d\in \R^n.
\]
By \cite[Theorem~2.35 and Proposition~2.37]{dhara2011optimality} and
\cite[Example~V.2.3.1]{hiriart1996convex}, this leads to
\[
\sigma_{\partial^{CL}\gh(\ov{x})}(d)+\sigma_{N_C(\ov{x})}(d)\geq 0,
\qquad\forall d\in \R^n,
\]
where $N_C(\ov{x}):=\{d\in \R^n \mid \langle d, x-\ov{x}\rangle\leq 0,~\forall x\in C\}$ is 
the \textit{normal cone} to $C$ at $\ov{x}$. From \cite[Theorem~V.3.3.3]{hiriart1996convex},
$\sigma_{\partial^{CL}\gh(\ov{x})}+\sigma_{N_C(\ov{x})}$ is the support function of
$\partial^{CL}\gh(\ov{x})+N_C(\ov{x})$. Hence, by \cite[Theorem~V.2.2.2]{hiriart1996convex},
$0\in \partial^{CL}\gh(\ov{x})+N_C(\ov{x})$. 
Furthermore, since $\partial^{CL}\gh(\ov{x})= \bs{{\rm conv}}\left(\partial^{M}\gh(\ov{x})\right)$, where
$\partial^{M}\gh(\ov{x})$ denotes the \textit{Mordukhovich subdifferential} of $\gh$ at $\ov{x}$
\cite[Exercise~4.36]{Mordukhovich2018}, we obtain
$0\in \bs{{\rm conv}}\left(\partial^{M}\gh(\ov{x})\right)+N_C(\ov{x})$. 
Thus, under local Lipschitz continuity, the stationarity condition in Definition~\ref{def:stationary} implies classical optimality conditions involving the Clarke and Mordukhovich sub\-di\-ffe\-ren\-tials. 
Notably, Definition~\ref{def:stationary} does not require local Lipschitz continuity, making it applicable to a broader class of functions.
\end{remark}

In general, stationarity as defined in Definition~\ref{def:stationary} does not imply minimality, even for smooth functions. For example, consider $\gh(x) = x^3$ and $C = [-1,1]$. At $\ov{x} = 0$, we have 
$ \gh^{D-} (\ov{x}; d)=\langle \nabla \gh(\ov{x}) , d\rangle\geq 0$ for all $d\in \R$, while  $\ov{x}=0$ is not a minimizer. Motivated by this example, we introduce a definition of a saddle point grounded in the notion of stationarity. A brief discussion of this definition, along with potential directions for its extension, is provided in Section~\ref{sec:conclusion}.

\begin{definition}[Saddle point]\label{def:saddlepoint}
Let $\gh: \R^n\to \Rinf$ be a proper function and $C\subseteq \R^n$ a closed convex set.
A stationary point $\ov{x}\in \dom{\gh}\cap C$ of $\gh$ on $C$ is called a \textit{saddle point} if $\ov{x}$ is neither a local minimizer nor a local maximizer.
\end{definition}

The subsequent theorem studies the existence, uniqueness, and optimality of stationary points under uniform quasiconvexity. Let us recall that we provided a sufficient condition for the $m$-supercoercivity of uniformly quasiconvex functions defined over the entire space $\R^n$ in Theorem~\ref{th:unquaissup}. However, as mentioned earlier, rather than imposing such conditions throughout the remainder of this paper, we directly assume that $\gf$ is coercive. This assumption is particularly advantageous from a numerical standpoint, as it enables us to handle coercive functions like $\gf(x)=\Vert x\Vert^q$, which are uniformly quasiconvex only on bounded subsets when $q \in (0,2)$.

\begin{theorem}[Stationarity and global optimality]\label{th:unifisdir}
 Let $\gf: \R^n\to \Rinf$ be a proper, uniformly quasiconvex function  with modulus $\phi$ on a closed convex set
 $C \subseteq \dom{\gf}$. 
 The following statements hold:
 \begin{enumerate}[label=(\textbf{\alph*}), font=\normalfont\bfseries, leftmargin=0.7cm]
\item\label{th:unifisdir:a} If  $\ov{x}\in \bs{\rm Dcrit}(\gf, C)$, then $\argmin{x \in C} \gf(x)=\{\ov{x}\}$, i.e., $\ov{x}$ is the unique global minimizer of $\gf$ on $C$. Moreover, for all $y \in C$,
\begin{equation}\label{growth:up2}
 \gf(\ov{x}) + \frac{1}{4} \phi(\lVert y - \ov{x}\rVert) \leq \gf(y);
\end{equation} 
\item \label{th:unifisdir:b} If $\gf$ is lsc and coercive, then
$\argmin{x \in C} \gf(x)=\{\ov{x}\}$ is a singleton, and relation \eqref{growth:up2} holds.
 \end{enumerate}
\end{theorem}
\begin{proof}
\ref{th:unifisdir:a}
Suppose $\ov{x}\in \bs{\rm Dcrit}(\gf, C)$.
Assume, for contradiction, that $\ov{x}$ is not a global minimizer, so there exists $\widehat{x}\in C$ with $\gf(\widehat{x})<\gf(\ov{x})$. By uniform quasiconvexity, for all $\lambda\in (0, 1)$, 
\begin{equation}\label{eq:th:unifisdir}
  \gf(\ov{x}+\lambda (\widehat{x}-\ov{x}))+\lambda(1-\lambda)\phi(\Vert\widehat{x}-\ov{x}\Vert)\leq \gf(\ov{x}).
\end{equation}
 Set $d=\widehat{x}-\ov{x}$ and $\beta=\phi(\Vert\widehat{x}-\ov{x}\Vert)>0$. 
For $\lambda\in \left(0, \frac{1}{2}\right)$,
\[
 \gf(\ov{x}+\lambda d)+\frac{\lambda}{2}\beta<
  \gf(\ov{x}+\lambda d)+\lambda(1-\lambda)\beta\leq \gf(\ov{x}),
\]
implying
\[
\gf^{D-} (\ov{x}; d)=\mathop{\bs\liminf}\limits_{\lambda \downarrow 0} \frac{\gf(\ov{x} + \lambda d) - \gf(\ov{x})}{\lambda}\leq -\frac{1}{2}\beta<0.
\]
Since $C$ is convex, $d \in D_C(\ov{x})$, contradicting $\ov{x}\in \bs{\rm Dcrit}(\gf, C)$. Hence,
$\ov{x}$ is a global minimizer. Furthermore, since
uniform quasiconvexity implies strict quasiconvexity, $\ov{x}$ is the unique minimizer. For the growth condition, let $\{\ov{x}\} = \argmin{x \in C} \gf(x)$. For any $y \in C$ and 
$\lambda \in [0, 1]$, we obtain
\begin{align*}
 \gf(\ov{x}) \leq \gf(\lambda \ov{x} + (1-\lambda)y) & \leq \bs\max\{\gf(\ov{x}),
 \gf(y)\} - \lambda (1-\lambda) \phi(\lVert y - \ov{x}\rVert) \\
 & = \gf(y) - \lambda (1-\lambda) \phi(\lVert y - \ov{x}\rVert).
\end{align*}
Setting $\lambda = \frac{1}{2}$ yields the desired inequality.
\\
\ref{th:unifisdir:b} Since $\gf$ is lsc and coercive, its sublevel sets $\mathcal{L}(\gf, \eta)$ are compact for each $\eta\in \R$
\cite[Proposition~11.12]{Bauschke17}. Thus, $\argmin{x \in C} \gf(x)\neq \emptyset$
\cite[Theorem~11.10]{Bauschke17}. Similarly to Assertion~\ref{th:unifisdir:a},  
 $\argmin{x \in C} \gf(x)$ is a singleton and one can obtain \eqref{growth:up2}.
\end{proof}

\begin{remark}\label{rem:clandlow}
 \begin{enumerate}[label=(\textbf{\alph*}), font=\normalfont\bfseries, leftmargin=0.7cm]
 \item \label{rem:clandlow:a} Theorem~\ref{th:unifisdir}~\ref{th:unifisdir:a}
 establishes that the set of stationary points coincides with the set of global minimizers for uniformly quasiconvex functions. For a locally Lipschitz, strongly quasiconvex function $\gf$, \cite[Lemma~1]{VIAL1982187} proves that $0\in\partial^{CL}\gf(\ov{x})$ if and only if $\ov{x}$ is the unique minimizer of $\gf$ on $\R^n$. By adap\-ting the arguments in \cite[Lemma~1]{VIAL1982187}, this result extends naturally to the class of uniformly quasiconvex functions. Consequently, such functions have no stationary points in the Clarke or Mordukhovich sense that are not global minimizers. See Section~\ref{sec:conclusion} for a discussion of the constrained case with nonconvex sets.

\item  The growth condition in Theorem~\ref{th:unifisdir}~\ref{th:unifisdir:b}, e.g., with $\phi(t)=\frac{\rho}{2}t^2$, yields a quadratic lower bound (see \cite[Proposition 34]{Kab-Lara-2})
$
\gf(\ov{x}) + \frac{\rho}{8} \lVert y - \ov{x}\rVert^2 \leq \gf(y).
$
This quantifies the steepness of the landscape around the unique minimizer $\ov{x}$ that
is a benefit for high-order proximal iterations as it ensures that the function value rises significantly away from the minimum, preventing iterates from stalling in flat regions and accelerating convergence to the global solution, a direct benefit of the benign landscape that enhances their practical utility in challenging optimization tasks.
\end{enumerate}
\end{remark}

We conclude this section by establishing explicit error bound conditions for uniformly quasiconvex functions.
These bounds quantify how the distance to the unique global minimizer can be controlled either by function values or by gradient norms.
Such estimates play a central role in convergence analysis, as they provide a direct link between stationarity measures and optimality gaps. 

\begin{theorem}[Error bound conditions]\label{th:errbound}
Let $\gf: \R^n\to \Rinf$ be a proper, uniformly quasiconvex function  with modulus $\phi$ and $\argmin{x \in \R^n} \gf(x)=\{\ov{x}\}$. The following statements hold:
  \begin{enumerate}[label=(\textbf{\alph*}), font=\normalfont\bfseries, leftmargin=0.7cm]
  \item \label{th:errbound:a} if $ \phi(t) \geq \rho t^q $ for some $ \rho > 0 $ and $q\geq 1$, then
   \[
 \lVert x - \ov{x}\rVert \leq \left(\frac{4}{\rho}(\gf(x)-\gf(\ov{x}))\right)^{\frac{1}{q}}, \qquad \forall x\in \R^n;
\]
\item \label{th:errbound:b} if $\gf$ is continuously differentiable and $ \phi(t) \geq \rho t^q $ for some $ \rho > 0 $ and $q> 1$, then
\[
 \Vert x-\ov{x}\Vert\leq \left(\frac{1}{\rho}\Vert \nabla\gf(x)\Vert\right)^{\frac{1}{q-1}}, \qquad \forall x\in \R^n, x\neq \ov{x}.
\]
  \end{enumerate}
\end{theorem}
\begin{proof}
\ref{th:errbound:a} By \eqref{growth:up2} and the assumption, we have 
$\frac{\rho}{4} \lVert x - \ov{x}\rVert^q \leq \gf(x)-\gf(\ov{x})$ for each $x\in\R^n$, which yields the desired inequality.
\\
\ref{th:errbound:b} Fix any $x\neq \ov{x}$. Since $\ov{x}$ is the (unique) global minimizer, Theorem \ref{th:chardiff}~\ref{th:chardiff:a} and the Cauchy--Schwarz inequality imply
\[
 \phi(\Vert x-\ov{x}\Vert) \leq \langle\nabla\gf(x), x-\ov{x}\rangle\leq\Vert \nabla\gf(x)\Vert \Vert x-\ov{x}\Vert.
\]
This implies
\[
\rho \Vert x-\ov{x}\Vert^q\leq \Vert \nabla\gf(x)\Vert \Vert x-\ov{x}\Vert.
\]
This completes the proof.
\end{proof}

The inequalities in Theorem~\ref{th:errbound} provide \textit{global H\"{o}lder-type error bounds} and gradient dominance conditions induced solely by uniform quasiconvexity.

These results eliminate the risk of convergence to suboptimal solutions or sa\-ddle points.
They form a cornerstone for the convergence guarantees of high-order proximal-point methods, as detailed in Section~\ref{sec:hoppaexact}.


\section{High-order proximal-point algorithm}
\label{sec:hoppaexact}
Building on the properties of uniformly quasiconvex functions established in Section~\ref{sec:ben}, in particular the fact that stationary points are global minimizers and that high-order proximal operators are well defined, as demonstrated subsequently, we propose a high-order proximal-point algorithm (HiPPA) for solving \eqref{eq:mainproblem}.
 This method leverages the unique global mi\-ni\-mizer and the favorable properties of the high-order proximal operator (HOPE) to ensure robust convergence, making it particularly effective for nonsmooth, nonconvex problems. 

We begin by recalling the high-order proximal operator, introduced in \eqref{eq:hippa}, and defining the high-order Moreau envelope, both central to our analysis.

\begin{definition}[High-order proximal operator and Moreau envelope]\label{def:Hiorder-Moreau env}
Let $p>1$, $\gamma>0$, $\gf: \R^n \to \Rinf$ be a proper function, and $C\subseteq\dom{\gf}$ a closed convex set. 
    The \textit{high-order proximal operator} (\textit{HOPE}) of $\gf$ and $C$ with parameter $\gamma$, denoted $\prox{\gf}{\gamma}{p}(C, \cdot): \R^n \rightrightarrows C$, is defined as 
   \begin{equation}\label{eq:Hiorder-Moreau prox}
       \prox{\gf}{\gamma}{p} (C, x):=\argmint{y\in C} \left(\gf(y)+\frac{1}{p\gamma}\Vert x- y\Vert^p\right),\tag{HOPE}
    \end{equation}     
    and the \textit{high-order Moreau envelope} (\textit{HOME}) of $\gf$ and $C$ with parameter $\gamma$, denoted $\fgam{\gf}{p}{\gamma}(C, \cdot):\R^n\to \R\cup\{\pm \infty\}$, 
    is given by
    \begin{equation}\label{eq:Hiorder-Moreau env}
    \fgam{\gf}{p}{\gamma}(C, x):=\mathop{\bs{\inf}}\limits_{y\in C} \left(\gf(y)+\frac{1}{p\gamma}\Vert x- y\Vert^p\right).\tag{HOME}
    \end{equation}
\end{definition}
When $C=\R^n$, we simplify the notation to $\prox{\gf}{\gamma}{p} (x)$ and $\fgam{\gf}{p}{\gamma}(x)$, respectively. 
The following facts and definitions establish key properties of HOPE and HOME; see \cite{Kabgani24itsopt,Kabganidiff} for more details.

\begin{fact}[Domain and majorizer for HOME]\cite[Proposition~12.9]{Bauschke17}\label{fact:horder:Bauschke17:p12.9}
Let $p> 1$ and $\gf:\R^n \to \Rinf$ be a proper function. Then, for every $\gamma > 0$,
$\dom{\fgam{\gf}{p}{\gamma}}=\R^n$. 
Moreover, for all $\gamma_2>\gamma_1>0$ and $x\in C$, 
$\fgam{\gf}{p}{\gamma_2}(C, x)\leq \fgam{\gf}{p}{\gamma_1}(C, x)\leq \gf(x)$.
\end{fact}

The next definition of uniform level boundedness \cite{Rockafellar09} is useful for proving some fundamental properties of HOME and HOPE in the upcoming sections.
\begin{definition}[Uniform level boundedness]\label{def:lboundlunif}\cite[Definition~1.16]{Rockafellar09}
A function $\Phi: \R^n \times \R^m\to \Rinf$, with values $\Phi(u, x)$, is \textit{level-bounded in $u$ locally uniform in $x$} if, for each $\ov{x}\in \R^m$ and $\lambda\in \R$, there exists a neighborhood $V$ of $\ov{x}$ along with a bounded set $B\subseteq \R^n$ such that
$\{u\mid \Phi(u, x)\leq \lambda\}\subseteq B$ for all $x\in V$.
\end{definition}

\begin{fact}[Basic properties of HOME and HOPE]\label{th:level-bound+locally uniform}\cite[Theorem~3.4]{Kabgani24itsopt}
Let $p>1$ and $\gf: \R^n\to \Rinf$ be a proper lsc function that is lower bounded. Then, for each $\gamma>0$,
\begin{enumerate}[label=(\textbf{\alph*}), font=\normalfont\bfseries, leftmargin=0.7cm]
\item \label{level-bound+locally uniform:levelunif} the function $\Phi(y, x):=\gf(y)+\frac{1}{p\gamma}\Vert x- y\Vert^p$  is
 level-bounded in $y$ locally uniformly in $x$ and lsc in $(y, x)$;
 
\item \label{level-bound+locally uniform2:con} $\fgam{\gf}{p}{\gamma}$ depends continuously on $(x,\gamma)$ in $\R^n\times (0, +\infty)$.
\end{enumerate}
 \end{fact} 

We present the definition of pro\-xi\-mal fixed points as the fixed points of the pro\-xi\-mal operator.
\begin{definition}[Proximal fixed point]\label{def:critic}
Let $p>1$, $\gh: \R^n \to \Rinf$ be a proper lsc function, and $C\subseteq\dom{\gh}$ a closed convex set. 
A point $\ov{x}\in \dom{\gh}\cap C$ is a \textit{proximal fixed point}
if $\ov{x}\in \prox{\gh}{\gamma}{p}(C,\ov{x})$. The set of all such points is denoted by
$\bs{\rm Fix}(\prox{\gh}{\gamma}{p}(C))$.
\end{definition}


We continue the study of HOPE and HOME, as well as the relationships among reference points induced by these operators. These properties ensure that HOPE is well-defined for uniformly quasiconvex functions, supporting the practical implementation and convergence of the HiPPA introduced in Algorithm~\ref{alg:HiPPA}.

\begin{theorem}[Further properties of HOME and HOPE]\label{th:nonemprox}
Let $\gf: \R^n\to \Rinf$ be a proper, lsc, coercive, and uniformly quasiconvex function with modulus $\phi$ in a closed convex set $C \subseteq \dom{\gf}$. The following statements hold:
\begin{enumerate}[label=(\textbf{\alph*}), font=\normalfont\bfseries, leftmargin=0.7cm]
\item \label{th:nonemprox:a} For each $\ov{x}\in \R^n$, $\gamma > 0$, and $p>1$, the set $\prox{\gf}{\gamma}{p} (C, \ov{x})$ is nonempty and compact;
\item \label{th:nonemprox:b} For each $\gamma > 0$ and $p>1$, $\bs\inf_{x\in C}  \fgam{\gf}{p}{\gamma}(x)=\bs\inf_{x\in C} \gf(x)$;
\item \label{th:nonemprox:c} For each $\gamma > 0$ and $p>1$,
\[
\argmint{x\in C}\fgam{\gf}{p}{\gamma}(x) =
\argmint{x\in C}\gf(x)
= \bs{\rm Fix}(\prox{\gf}{\gamma}{p}(C))=\bs{\rm Dcrit}(\gf, C);
\]
\item \label{th:nonemprox:d} 
 For a sequence $\{x^k\}_{k\in \mathbb{N}}\subseteq C$, if 
 $y^k \in \prox{\gf}{\gamma_k}{p}(C, x^k)$, with $x^k\to \ov{x}$ and 
 $\gamma_k \to \gamma$,  the sequence $\{y^k\}_{k\in \mathbb{N}}$ is bounded, and all its cluster points lie in $\prox{\gf}{\gamma}{p} (C, \ov{x})$.
\end{enumerate}
\end{theorem}
\begin{proof}
\ref{th:nonemprox:a}
The function $x\mapsto \gf(x)+\frac{1}{p\gamma}\Vert x - \ov{x}\Vert^p$ is coercive and lsc. Thus, its sublevel sets are compact. By \cite[Corollary~1.10]{Rockafellar09}, 
\[
\argmint{y\in C} \left(\gf(y)+\frac{1}{p\gamma}\Vert \ov{x}- y\Vert^p\right)\neq \emptyset.
\]
Coercivity and lower semicontinuity further ensure that $\prox{\gf}{\gamma}{p} (C, \ov{x})$ is compact.
\\
\ref{th:nonemprox:b}
Invoking the definition of HOME and Fact~\ref{fact:horder:Bauschke17:p12.9}, for each $x\in C$ and $\gamma>0$, we arrive at
\[
\mathop{\bs\inf}\limits_{x\in C} \gf(x) \leq \fgam{\gf}{p}{\gamma} (C, x) \leq \gf(x),
\]
and therefore the infima coincide.
\\
\ref{th:nonemprox:c}
Let $\ov{x}\in \argmin{y\in C} \fgam{\gf}{p}{\gamma}(y)$ and
$\ov{y}\in \prox{\gf}{\gamma}{p} (C,\ov{x})$, which exist by Assertion~\ref{th:nonemprox:a}.
From Assertion~\ref{th:nonemprox:b},
\[
\gf(\ov{y})+\frac{1}{p\gamma}\Vert \ov{x}-\ov{y}\Vert^p=\fgam{\gf}{p}{\gamma}(\ov{x})=\mathop{\bs{\inf}}\limits_{y\in C}  \fgam{\gf}{p}{\gamma}(y)=\mathop{\bs{\inf}}\limits_{y\in C} \gf(y)\leq \gf(\ov{y})\leq \gf(\ov{y})+\frac{1}{p\gamma}\Vert \ov{x}-\ov{y}\Vert^p,
\]
i.e., $\gf(\ov{y})=\gf(\ov{y})+\frac{1}{p\gamma}\Vert \ov{x}-\ov{y}\Vert^p$.
As such, $\ov{x}=\ov{y}$ and $\ov{x}\in \argmin{y\in C} \gf(y)$, leading to
\[\argmint{y\in C} \fgam{\gf}{p}{\gamma}(y)\subseteq \argmint{y\in C} \gf(y).\]
The inclusion $\argmin{x\in C}\gf(x)\subseteq \bs{\rm Fix}(\prox{\gf}{\gamma}{p}(C))$ follows from the definitions.
For $\ov{x}\in \bs{\rm Fix}(\prox{\gf}{\gamma}{p}(C))$, we have
$\ov{x}\in \argmin{y\in C} \left(\gf(y)+\frac{1}{p\gamma}\Vert \ov{x}- y\Vert^p\right)$, which implies that for each $d \in D_C(\ov{x})$, 
\[\left(\gf(\cdot)+ \frac{1}{p\gamma}\Vert \ov{x}- \cdot\Vert^p\right)^{D-}(\ov{x}; d) \geq 0.\]
Since $p > 1$, together with \eqref{eq:sumruldini}, this implies that for each $d \in D_C(\ov{x})$, $\gf^{D-}(\ov{x}; d) \geq 0$. Hence, 
$\bs{\rm Fix}(\prox{\gf}{\gamma}{p}(C)) \subseteq \bs{\rm Dcrit}(\gf, C)$.
  Theorem \ref{th:unifisdir} establishes
$\bs{\rm Dcrit}(\gf, C) \subseteq \argmin{x \in C} \gf(x)$. 
Finally, for $y\in \argmin{x\in C}\gf(x)$ and any $\ov{x}\in C$, 
\[
\gf(y)\leq \gf(x)+\frac{1}{p\gamma}\Vert x - \ov{x}\Vert^p,\qquad \forall x\in C,
\]
i.e., $\gf(y)\leq \fgam{\gf}{p}{\gamma}(C, \ov{x})$, meaning that $\gf(y)=\fgam{\gf}{p}{\gamma}(C, y)$ leads to
$y\in \argmin{x\in C}\fgam{\gf}{p}{\gamma}(x)$.
\\
\ref{th:nonemprox:d}
Coercivity and lower semicontinuity imply that $\gf$ is bounded from below on $\R^n$
\cite[Theorem 1.9]{Rockafellar09}, satisfying the assumptions of Fact~\ref{th:level-bound+locally uniform}. The result follows from \cite[Theorem~1.17~(b)]{Rockafellar09} and the closedness of $C$.
\end{proof}

\subsection{HiPPA and its convergence analysis}
In this section, we present the high-order proximal-point algorithm (HiPPA), which iteratively applies HOPE with parameters 
$\gamma_k\in (\gamma_{\bs\min}, \gamma_{\bs\max})$, where $\gamma_{\bs\max}>\gamma_{\bs\min}>0$,
to generate a sequence that will be shown to converge to the global minimizer of $\gf$. The algorithm leverages the nonemptiness and outer semicontinuity of HOPE established in Theorem~\ref{th:nonemprox} to ensure robust convergence for uniformly quasiconvex functions.
The following is its formal description.
\begin{algorithm}[H]
\caption{HiPPA (High-order Proximal-Point Algorithm)}\label{alg:HiPPA}
\begin{algorithmic}[1]
\STATE Choose $x^0\in C$, $\gamma_{\bs\max}>\gamma_{\bs\min}>0$ and $\gamma_0\in (\gamma_{\bs\min},\gamma_{\bs\max})$. Set $k=0$;
\WHILE{stopping criteria do not hold}
\STATE Select $\gamma_k\in (\gamma_{\bs\min},\gamma_{\bs\max})$ and compute $x^{k+1}\in \prox{\gf}{\gamma_k}{p} (C, x^k)$;
\STATE Set $k = k + 1$.
\ENDWHILE
\end{algorithmic}
\end{algorithm}
In Algorithm~\ref{alg:HiPPA}, $\gamma_k$ controls the regularization strength in the proximal step, balancing the influence of the objective $\gf$ and the regularizer.
The following theorem indicates convergence to the global minimizer.

\begin{theorem}[Global convergence]\label{prop:propertiesppa}
 Let $\gf: \R^n\to \Rinf$ be a proper, lsc, coercive, and uniformly quasiconvex function with modulus $\phi$ in a closed convex set $C \subseteq \dom{\gf}$.
 For $p > 1$, if $\{x^k\}_{k \in \Nz}$ is generated by Algorithm~\ref{alg:HiPPA}, the following statements hold:
\begin{enumerate}[label=(\textbf{\alph*}), font=\normalfont\bfseries, leftmargin=0.7cm]
 \item \label{prop:propertiesppa:noninc} The sequences $\{\gf(x^k)\}_{k \in \Nz}$ and $\{\fgam{\gf}{p}{\gamma_k}(C, x^k)\}_{k \in \Nz}$ are decreasing;
  \item \label{prop:propertiesppa:Cauchy} $\sum_{k=0}^{\infty} \Vert x^{k+1} - x^k \Vert^p < \infty$;
  \item \label{prop:propertiesppa:bound} The sequence $\{x^k\}_{k \in \Nz}$ is bounded;
  \item \label{prop:propertiesppa:con} The set of cluster points satisfies $\Omega(x^k)=\argmin{x\in C}\gf(x)$;
  \item \label{prop:propertiesppa:glob} The sequence $\{x^k\}_{k \in \Nz}$ converges to $\{\ov{x}\} = \argmin{x \in C} \gf(x)$.
\end{enumerate}
\end{theorem}

\begin{proof}
\ref{prop:propertiesppa:noninc}
If $x^{k+1} = x^{k}$, then by Theorem~\ref{th:nonemprox}~\ref{th:nonemprox:c}, 
$x^{k} \in \argmint{x\in C}\fgam{\gf}{p}{\gamma}(x)$ and the algorithm will stop. 
Assume $x^{k+1} \neq x^{k}$. 
Since $x^{k+1} \in \prox{\gf}{\gamma_k}{p}(C, x^k)$, by Fact \ref{fact:horder:Bauschke17:p12.9}, it holds that
\begin{equation}\label{eq1:prop:propertiesppa}
 \gf(x^{k+1})<\gf(x^{k+1})+\frac{1}{p\gamma_k}\Vert x^{k+1}-x^k\Vert^p=\fgam{\gf}{p}{\gamma_k}(C, x^k)\leq \gf(x^k),
\end{equation}
and
\begin{equation}\label{eq2:prop:propertiesppa}
 \fgam{\gf}{p}{\gamma_{k+1}}(C, x^{k+1})\leq \gf(x^{k+1})=\fgam{\gf}{p}{\gamma_k}(C, x^k)-\frac{1}{p\gamma_k}\Vert x^{k+1}-x^k\Vert^p< \fgam{\gf}{p}{\gamma_k}(C, x^k),
\end{equation}
guaranteeing the decreasing property for $\{\gf(x^k)\}_{k \in \Nz}$ and $\{\fgam{\gf}{p}{\gamma_k}(C, x^k)\}_{k \in \Nz}$.
\\
\ref{prop:propertiesppa:Cauchy} In view of \eqref{eq1:prop:propertiesppa}, we get
\begin{equation}\label{eq:sumonseq}
\sum_{k=0}^{\infty}\Vert x^{k+1} - x^k\Vert^p\leq p\gamma_{\bs\max}(\gf(x^0)-\mathop{\bs\inf}\limits_{x\in C}\gf(x))<\infty.
\end{equation}
\ref{prop:propertiesppa:bound}
From Assertion~\ref{prop:propertiesppa:noninc}, $\{x^k\}_{k\in \Nz}\subseteq\mathcal{L}(\gf,\gf(x^0))$ and from coercivity $\mathcal{L}(\gf,\gf(x^0))$ is bounded, which proves the claim.
\\
\ref{prop:propertiesppa:con}
 From Assertion~\ref{prop:propertiesppa:Cauchy}, we have $\Vert x^{k+1} - x^k \Vert \to 0$. Note that $x^{k+1} \in \prox{\gf}{\gamma_k}{p}(C, x^k)$. We define $y^k := x^{k+1}$ to highlight this. Thus, the sequences $\{x^k\}_{k \in \Nz}$ and $\{y^k\}_{k \in \Nz}$ share the same cluster points. 
Let $\widehat{x} \in \Omega(x^k)$, and let $\{x^j\}_{j \in J \subseteq \Nz}$ be a subsequence such that $x^j \to \widehat{x}$. For the corresponding subsequence 
$\{y^j\}_{j \in J}$, we also have $y^j \to \widehat{x}$. From the boundedness of $\{\gamma_j\}_{j \in J \subseteq \Nz}$,
We assume that there exists some $\gamma_b$ such that $\gamma_k\to \gamma_b$. Otherwise, we can work with a subsequence of $\{x^j\}_{j \in J \subseteq \Nz}$. By Theorem~\ref{th:nonemprox}~\ref{th:nonemprox:d}, we get
$\widehat{x} \in \prox{\gf}{\gamma_b}{p}(C, \widehat{x})$. Thus, $\widehat{x}$ is a proximal fixed point. From Theorem~\ref{th:nonemprox}~\ref{th:nonemprox:c}, we conclude that $\widehat{x} \in \argmin{x \in C} \gf(x)$. Now, Theorem~\ref{th:unifisdir}~\ref{th:unifisdir:b} implies that $\Omega(x^k) = \argmin{x \in C} \gf(x)$.
\\
\ref{prop:propertiesppa:glob} It follows from Theorem~\ref{th:unifisdir}~\ref{th:unifisdir:b} and Assertions~\ref{prop:propertiesppa:bound}~and~\ref{prop:propertiesppa:con}.
\end{proof}

We note that other high-order approaches based on tensor models and cubic-quartic regularization for problems with higher-order derivatives have been studied in the literature (e.g., \cite{Cartis25Cubic,Cartis26second,zhu2024global}); however, these methods are fundamentally different from the proximal-point framework considered here, as they rely on high-order Taylor models rather than high-order proximal regularization.

\subsection{Convergence rate analysis}
\label{sec:conv}
Having introduced the high-order proximal-point algorithm, we now analyze its convergence behavior for uniformly quasiconvex functions. As stationary points are unique global minimizers, it suggests that Algorithm~\ref{alg:HiPPA} should converge reliably; however, the convergence rates depend on the function properties (e.g., strong vs. uniform quasiconvexity) and the proximal order $p$. This section quantifies these rates, extending the classical results for strongly convex cases (e.g., \cite{rockafellar1976monotone}) to the broader class of uniformly quasiconvex functions.

\begin{lemma}[Key recursive inequalities]\label{th:con:rel}
 Let $\gf: \R^n\to \Rinf$ be a proper, lsc, coercive, and uniformly quasiconvex function with modulus $\phi$ on a closed convex set $C \subseteq \dom{\gf}$.
If the sequence $\{x^{k}\}_{k\in \Nz}$ is generated by Algorithm~\ref{alg:HiPPA} and 
$\{\ov{x}\}=\argmin{x \in C} \gf(x)$, then the fo\-llo\-wing statements hold:
\begin{enumerate}[label=(\textbf{\alph*}), font=\normalfont\bfseries, leftmargin=0.7cm]
\item \label{th:con:rel:p<2} if $p\in (1, 2)$, there exists a constant $\sigma_p>0$ such that
\begin{equation}\label{eq:th:con:stqconvexp2}
\gamma_{k}\frac{\phi( \lVert \ov{x} - x^{k+1} \rVert )}{\lVert\ov{x}-x^{k+1}\rVert} +
\frac{\sigma_p}{p}\lVert \ov{x} - x^{k+1} \rVert\leq
   \Vert \ov{x}-x^{k}\Vert^{p-1}, \qquad \forall k\in \Nz;
\end{equation}
\item \label{th:con:rel:p=2} if $p=2$,
\begin{equation}\label{eq:th:con:stqconvexp=2}
 \gamma_{k}\frac{\phi( \lVert\ov{x} - x^{k+1} \rVert )}{\lVert  \ov{x}- x^{k+1} \rVert} 
  + \lVert \ov{x}- x^{k+1} \rVert \leq  \lVert \ov{x}- x^{k}\rVert, \qquad \forall k\in \Nz;
\end{equation}
\item \label{th:con:rel:p>2} if $p> 2$, 
\begin{equation}\label{eq:th:con:stqconvex}
\gamma_{k}\frac{\phi( \lVert \ov{x} - x^{k+1} \rVert )}{\lVert \ov{x} - x^{k+1} \rVert}
 +\frac{\widehat{\sigma}_p}{p} \lVert \ov{x} - x^{k+1} \rVert^{p-1}\leq
 \Vert \ov{x}-x^{k}\Vert^{p-1},\qquad \forall k\in \Nz, 
\end{equation}
where $\widehat{\sigma}_p=\left(\frac{1}{2}\right)^\frac{3p-2}{2}$.
\end{enumerate}
\end{lemma}

\begin{proof}
From the definition, for each $p>1$,
\begin{equation}\label{eq1:th:con:stqconvex}
 \gf (x^{k+1}) + \frac{1}{p \gamma_{k}} \lVert x^{k+1} - x^{k} \rVert^{p}\leq \gf (x) 
 + \frac{1}{p \gamma_{k}} \lVert x - x^{k} \rVert^{p}, \qquad \forall ~ x \in C.
\end{equation}
 For each $\lambda \in [0, 1]$, we have
$x_\lambda := \lambda \ov{x} + (1-\lambda)x^{k+1} \in C$. Now, we prove the assertions.
\\
\ref{th:con:rel:p<2}
Let $p\in (1, 2)$ be arbitrary. From Theorem~\ref{prop:propertiesppa}~\ref{prop:propertiesppa:bound} and Example~\ref{ex:norm:unif}, by considering a large enough ball including $\ov{x}$ and $\{x^{k}\}_{k\in \Nz}$,  there exists some $\sigma_p>0$ such that for each $k\in \Nz$, from \eqref{eq1:th:con:stqconvex},
 \begin{align*}
 \gf (x^{k+1}) &+ \frac{1}{p \gamma_{k}} \lVert x^{k+1} - x^{k} \rVert^{p}\\& \leq
 \gf (\lambda \ov{x} + (1-\lambda) x^{k+1}) + \frac{1}{p \gamma_{k}} \lVert \lambda 
  (\ov{x} - x^{k}) + (1-\lambda)(x^{k+1}-x^{k}) \rVert^{p} \notag \\
  &\leq \bs\max\{\gf (\ov{x}), \gf (x^{k+1})\} - \lambda (1 - \lambda) \phi( \lVert 
  \ov{x} - x^{k+1} \rVert ) + \frac{\lambda}{p \gamma_{k}} \lVert \ov{x} - x^{k} \rVert^{p}
  \notag \\
  & ~~~~ +  \frac{(1-\lambda)}{p \gamma_{k}} \lVert x^{k+1}-x^{k} \rVert^{p} - 
  \frac{\lambda (1-\lambda)\sigma_p}{p \gamma_{k}} \lVert \ov{x} - x^{k+1} \rVert^{2}. \notag
\end{align*} 
Then, it follows from $\bs\max\{\gf (\ov{x}), \gf (x^{k+1})\}=\gf (x^{k+1})$ that
\[
  \lambda (1 - \lambda) \phi( \lVert \ov{x} - x^{k+1} \rVert ) +\frac{\lambda (1-\lambda){\sigma}_p}{p \gamma_{k}} \lVert \ov{x} - x^{k+1} \rVert^{2}\leq
  \frac{\lambda}{p \gamma_{k}} \lVert \ov{x} - x^{k} \rVert^{p}
  -  \frac{\lambda}{p \gamma_{k}} \lVert x^{k+1}-x^{k} \rVert^{p}. 
\]
By dividing both sides by $\lambda$ and letting $\lambda\downarrow 0$, we get
\begin{equation}\label{eq2:th:con:stqconvexp2}
 p \gamma_{k}\phi( \lVert \ov{x} - x^{k+1} \rVert ) +\sigma_p \lVert \ov{x} - x^{k+1} \rVert^{2}\leq
   \lVert \ov{x} - x^{k} \rVert^{p}
  -   \lVert x^{k+1}-x^{k} \rVert^{p}. 
\end{equation}
In light of \eqref{eq:ineqp}, we get
\[
\Vert \ov{x}-x^{k}\Vert^p\leq \Vert x^{k+1}-x^{k}\Vert^p +p\Vert \ov{x}-x^{k}\Vert^{p-1}\lVert\ov{x}-x^{k+1}\rVert.
\]
Combining with \eqref{eq2:th:con:stqconvexp2}, this implies
\[
 p \gamma_{k}\phi( \lVert \ov{x} - x^{k+1} \rVert ) +\sigma_p \lVert \ov{x} - x^{k+1} \rVert^{2}\leq
   p\Vert \ov{x}-x^{k}\Vert^{p-1}\lVert\ov{x}-x^{k+1}\rVert,
\]
which yields                         
\eqref{eq:th:con:stqconvexp2}.
\\
\ref{th:con:rel:p=2}
From \eqref{eq:eqp=22} and \eqref{eq1:th:con:stqconvex}, for each 
$\lambda\in [0,1]$, we obtain
 \begin{align*}
   \gf (x^{k+1})& + \frac{1}{2 \gamma_{k}} \lVert x^{k+1} - x^{k} \rVert^{2}\\& \leq
 \gf (\lambda \ov{x} + (1-\lambda) x^{k+1}) + \frac{1}{2 \gamma_{k}} \lVert \lambda 
  (\ov{x} - x^{k}) + (1-\lambda)(x^{k+1}-x^{k}) \rVert^{2}, \notag \\
  &\leq \bs\max\{\gf (\ov{x}), \gf (x^{k+1})\} - \lambda (1 - \lambda) \phi( \lVert 
  \ov{x} - x^{k+1} \rVert ) + \frac{\lambda}{2 \gamma_{k}} \lVert \ov{x} - x^{k} \rVert^{2}
  \notag \\
  & ~~~~ +  \frac{(1-\lambda)}{2 \gamma_{k}} \lVert x^{k+1}-x^{k} \rVert^{2} - 
  \frac{\lambda (1-\lambda)}{2 \gamma_{k}} \lVert \ov{x} - x^{k+1} \rVert^{2}. \notag
\end{align*} 
Together with \eqref{eq:eqp=21}, this implies
 \begin{align*}
   \lambda (1 - \lambda) \phi( \lVert 
  \ov{x} - x^{k+1} \rVert )
  &\leq  \frac{\lambda}{2 \gamma_{k}} \lVert \ov{x} - x^{k} \rVert^{2}
 -  \frac{\lambda}{2 \gamma_{k}} \lVert x^{k+1}-x^{k} \rVert^{2} - 
  \frac{\lambda (1-\lambda)}{2 \gamma_{k}} \lVert \ov{x} - x^{k+1} \rVert^{2}
  \\& = \frac{\lambda}{\gamma_{k}} \langle x^{k+1} - x^{k}, \ov{x} - x^{k+1} \rangle +  \frac{\lambda^{2}}{2 \gamma_{k}} \lVert \ov{x} - x^{k+1} \rVert^{2}.
\end{align*} 
Dividing both sides by $\lambda$ and letting $\lambda\downarrow 0$ yield
 \begin{align*}
  \gamma_{k}\phi( \lVert 
  \ov{x} - x^{k+1} \rVert )\leq\langle x^{k+1} - x^{k}, \ov{x} - x^{k+1} \rangle
  \leq - \lVert \ov{x}-x^{k+1}\rVert^{2} + \lVert \ov{x}-x^{k}\rVert \lVert \ov{x}-x^{k+1}\rVert,
\end{align*} 
establishing \eqref{eq:th:con:stqconvexp=2}.
\\
\ref{th:con:rel:p>2}
Invoking Example~\ref{ex:norm:unif} and \eqref{eq1:th:con:stqconvex}, for each 
$\lambda\in [0,1]$ and $p> 2$, we arrive at
 \begin{align*}
   \gf (x^{k+1}) &+ \frac{1}{p \gamma_{k}} \lVert x^{k+1} - x^{k} \rVert^{p} \\&\leq
 \gf (\lambda \ov{x} + (1-\lambda) x^{k+1}) + \frac{1}{p \gamma_{k}} \lVert \lambda 
  (\ov{x} - x^{k}) + (1-\lambda)(x^{k+1}-x^{k}) \rVert^{p}, \notag \\
  &\leq\bs \max\{\gf (\ov{x}), \gf (x^{k+1})\} - \lambda (1 - \lambda) \phi( \lVert 
  \ov{x} - x^{k+1} \rVert ) + \frac{\lambda}{p \gamma_{k}} \lVert \ov{x} - x^{k} \rVert^{p}
  \notag \\
  & ~~~~ +  \frac{(1-\lambda)}{p \gamma_{k}} \lVert x^{k+1}-x^{k} \rVert^{p} - 
  \frac{\lambda (1-\lambda)\widehat{\sigma}_p}{p \gamma_{k}} \lVert \ov{x} - x^{k+1} \rVert^{p}, \notag
\end{align*} 
implying
\[
  \lambda (1 - \lambda) \phi( \lVert \ov{x} - x^{k+1} \rVert ) +\frac{\lambda (1-\lambda)\widehat{\sigma}_p}{p \gamma_{k}} \lVert \ov{x} - x^{k+1} \rVert^{p}\leq
  \frac{\lambda}{p \gamma_{k}} \lVert \ov{x} - x^{k} \rVert^{p}
  -  \frac{\lambda}{p \gamma_{k}} \lVert x^{k+1}-x^{k} \rVert^{p}. 
\]
By dividing both sides by $\lambda$ and letting $\lambda\downarrow 0$, we get
\begin{equation}\label{eq2:th:con:stqconvex}
 \phi( \lVert \ov{x} - x^{k+1} \rVert ) +\frac{\widehat{\sigma}_p}{p \gamma_{k}} \lVert \ov{x} - x^{k+1} \rVert^{p}\leq
  \frac{1}{p \gamma_{k}} \lVert \ov{x} - x^{k} \rVert^{p}
  -  \frac{1}{p \gamma_{k}} \lVert x^{k+1}-x^{k} \rVert^{p}. 
\end{equation}
Combining with \eqref{eq:ineqp}, this ensures 
\begin{align*}
\Vert \ov{x}-x^{k}\Vert^p&=\Vert (x^{k}-x^{k+1}) - (\ov{x}-x^{k+1})\Vert^p
\\&\leq \Vert x^{k+1}-x^{k}\Vert^p - p\Vert \ov{x}-x^{k}\Vert^{p-2}\langle x^{k}-\ov{x}, \ov{x}-x^{k+1}\rangle
\\&= \Vert x^{k+1}-x^{k}\Vert^p + p\Vert \ov{x}-x^{k}\Vert^{p-2}\langle x^{k}-\ov{x}, x^{k+1}-\ov{x}\rangle
\\&\leq \Vert x^{k+1}-x^{k}\Vert^p + p\Vert \ov{x}-x^{k}\Vert^{p-1}\lVert\ov{x}-x^{k+1}\rVert.
\end{align*} 
It follows from this and \eqref{eq2:th:con:stqconvex} that
\[
 \phi( \lVert \ov{x} - x^{k+1} \rVert ) +\frac{\widehat{\sigma}_p}{p \gamma_{k}} \lVert \ov{x} - x^{k+1} \rVert^{p}\leq
  \frac{1}{\gamma_k}\Vert \ov{x}-x^{k}\Vert^{p-1}\lVert\ov{x}-x^{k+1}\Vert,
\]
leading to \eqref{eq:th:con:stqconvex}.
\end{proof}

The following theorem provides a detailed asymptotic convergence rate analysis for HiPPA under various conditions on the uniform quasiconvexity modulus $\phi$, revealing scenarios where the algorithm achieves linear and superlinear convergence. The results highlight the interplay between the geometry of the problem (captured by $\phi$ and its degree $q$) and the structure of the algorithm (defined by $p$).

\begin{theorem}\label{th:convgstr}
Let $\gf: \R^n\to \Rinf$ be a proper, lsc, coercive, and uniformly quasiconvex function with modulus $\phi$ on a closed convex set $C \subseteq \dom{\gf}$.
 If the sequence $\{x^{k}\}_{k\in \Nz}$ is generated by Algorithm~\ref{alg:HiPPA} and 
$\{\ov{x}\}=\argmin{x \in C} \gf(x)$, then the following statements hold:
    \begin{enumerate}[label=(\textbf{\alph*}), font=\normalfont\bfseries, leftmargin=0.7cm]
    \item \label{th:convgstr:a} If for some $\rho_q>0$, we have $\rho_q t^q\leq \phi(t)$ on $[0,1)$ for some $q\in (1,2)$, then for each $p\in [q,2)$, by choosing $\gamma_{\bs\min}>\frac{1}{\rho_q}$ the convergence rate is locally linear, i.e., there exists $\ov{k}_p\in \Nz$ such that 
 \[
 \frac{\lVert\ov{x}-x^{k+1}\rVert }{\Vert \ov{x}-x^{k}\Vert}
\leq
   \left(\frac{1}{\rho_q\gamma_{\bs\min}}\right)^{\frac{1}{p-1}}<1, \qquad \forall k\geq \ov{k}_p.  
\]
     \item \label{th:convgstr:b} If for some $\rho>0$, we have $\rho t^2\leq \phi(t)$ on $[0,+\infty)$     
     and $p=2$, the convergence rate is linear, i.e., 
    \[
\frac{\lVert \ov{x} - x^{k+1} \rVert}{\lVert \ov{x} -x^{k}\rVert}
 \leq  \frac{1}{1+\gamma_{\bs\min}\rho}<1, \qquad \forall k\in \Nz.
\]

 \item \label{th:convgstr:d} If $p>2$, and for some $\rho_p>0$, we have $\rho_p t^p\leq \phi(t)$ on $[0,+\infty)$, then by choosing $\gamma_{\bs\min}>\frac{1}{\rho_p}$, the convergence rate is linear, i.e., 
    \[
 \frac{\lVert \ov{x} - x^{k+1} \rVert}{\Vert \ov{x}-x^{k}\Vert}
 \leq \left(\frac{p}{p\gamma_{\bs\min}\rho_p+\widehat{\sigma}_p}\right)^{\frac{1}{p-1}}<1
,\qquad \forall k\in \Nz, 
\]
where $\widehat{\sigma}_p=\left(\frac{1}{2}\right)^\frac{3p-2}{2}$.

 \item \label{th:convgstr:e} If  
  for some $q\geq 2$, there exists $\rho_q>0$, such that $\rho_q t^q\leq \phi(t)$ on $[0,+\infty)$, and
 $p>q$, then the convergence is $Q$-superlinear of order $\frac{p-1}{q-1}$. In particular, for the constant $c=(\gamma_{\bs\min}\rho_q)^{\frac{-1}{q-1}}$,
    \[
    \lVert x^{k+1} - \ov{x} \rVert\leq c \lVert x^{k} - \ov{x}\rVert^{\frac{p-1}{q-1}}, \qquad \forall k\in \Nz.
    \]
    \end{enumerate}
\end{theorem}

\begin{proof}
\ref{th:convgstr:a} 
Since $x^k\to\ov{x}$,
from Lemma~\ref{th:con:rel}~\ref{th:con:rel:p<2}, there exists $\ov{k}_p\in \Nz$ such that for each $k\geq \ov{k}_p$,
$\Vert \ov{x} - x^{k+1}\Vert<1$ and
\[
\rho_q\gamma_{k}\frac{\lVert \ov{x} - x^{k+1} \rVert^q}{\lVert\ov{x}-x^{k+1}\rVert} +
\frac{\sigma_p}{p}\lVert \ov{x} - x^{k+1} \rVert\leq
   \Vert \ov{x}-x^{k}\Vert^{p-1}, \qquad \forall k\geq \ov{k}_p.
\]
Since $\lVert \ov{x} - x^{k+1} \rVert^{p-1}\leq \lVert \ov{x} - x^{k+1} \rVert^{q-1}$, it holds that
\[
\rho_q\gamma_{k}\lVert \ov{x} - x^{k+1} \rVert^{p-1}\leq
   \Vert \ov{x}-x^{k}\Vert^{p-1}, \qquad \forall k\geq \ov{k}_p.  
\]
This ensures our desired result.
\\
\ref{th:convgstr:b} Invoking Lemma~\ref{th:con:rel}~\ref{th:con:rel:p=2} yields
\[
(1+\gamma_{k}\rho) \lVert x^{k+1} - \ov{x} \rVert
 \leq  \lVert x^{k} - \ov{x}\rVert, \qquad \forall k\in \Nz.
\]
establishing our claim. 
\\
\ref{th:convgstr:d} It follows from Lemma~\ref{th:con:rel}~\ref{th:con:rel:p>2}, for all $k\in \Nz$, that
\[
\left(\gamma_{\bs\min}\rho_p
 +\frac{\widehat{\sigma}_p}{p} \right)\lVert \ov{x} - x^{k+1} \rVert^{p-1}
\leq\gamma_{k}\frac{\phi( \lVert \ov{x} - x^{k+1} \rVert )}{\lVert \ov{x} - x^{k+1} \rVert}
 +\frac{\widehat{\sigma}_p}{p} \lVert \ov{x} - x^{k+1} \rVert^{p-1}\leq
 \Vert \ov{x}-x^{k}\Vert^{p-1},
\]
meaning that
\[
 \frac{\lVert \ov{x} - x^{k+1} \rVert}{\Vert \ov{x}-x^{k}\Vert}
 \leq \left(\frac{p}{p\gamma_{\bs\min}\rho_p+\widehat{\sigma}_p}\right)^{\frac{1}{p-1}}<1
,\qquad \forall k\in \Nz.
\]
\ref{th:convgstr:e}
Invoking Lemma~\ref{th:con:rel}~\ref{th:con:rel:p>2} ensures
\[
\gamma_{\bs\min}\rho_q \frac{\lVert \ov{x} - x^{k+1} \rVert^{q-1}}{\Vert \ov{x}-x^{k}\Vert^{p-1}}
  \leq
1,\qquad \forall k\in \Nz,
 \]
 establishing our claim.
\end{proof}

Theorem~\ref{th:convgstr} establishes several cases detailing the asymptotic convergence rate of the HiPPA sequence based on the characteristics of the uniform quasiconvexity modulus $\phi$ and the proximal order $p$. Table~\ref{tab:sum} provides a concise summary of these results.

\begin{table}[h]
\centering
\scalebox{0.9}{\begin{tabular}{cccc}
\toprule
\textbf{Modulus Condition} & \textbf{Order \( p \)} & \textbf{Convergence Rate} & \textbf{\( \gamma_{\min} \)} \\
\midrule
\( \phi(t) \geq \rho_q t^q \), \( q \in (1, 2] \), \( \rho_q > 0 \), on \( [0,1) \) & \( p \in [q, 2) \) & Locally linear & \( \gamma_{\min} > \frac{1}{\rho_q} \) \\
 \( \phi(t) \geq \rho t^2 \), \( \rho > 0 \), on \( [0, +\infty) \) & \( p = 2 \) & Linear & \( \gamma_{\min} > 0 \) \\
\( \phi(t) \geq \rho_q t^q \), \( \rho_q > 0 \), on \( [0, +\infty) \) & \( p=q > 2 \) & Linear& \( \gamma_{\min} > \frac{1}{\rho_p} \) \\
\( \phi(t) \geq \rho_q t^q \), \( q \geq 2 \), \( \rho_q > 0 \), on \( [0, +\infty) \) & \( p > q \) & Superlinear& \( \gamma_{\min} > 0 \) \\
\bottomrule
\end{tabular}}
\caption{Summary of the convergence rates discussed in Theorem~\ref{th:convgstr} for HiPPA}\label{tab:sum}
\end{table}

\begin{remark}
The inequality in Theorem~\ref{th:convgstr}~\ref{th:convgstr:e} is to be understood as a superlinear
\textit{asymptotic} estimate. In fact, since $\frac{p-1}{q-1}>1$, it implies
\[
\frac{\|x^{k+1}-\ov{x}\|}{\|x^k-\ov{x}\|}\to 0 \quad \text{as } k\to\infty.
\]
The estimate does not require monotonic decay of the error at every
iteration.
If one additionally wishes to guarantee a one-step contraction
$\|x^{k+1}-\ov{x}\|\le \|x^k-\ov{x}\|$, it suffices to assume
\[
\|x^0-\ov{x}\|\le c^{-\frac{q-1}{p-q}},
\]
which provides an explicit neighborhood where the mapping is contractive.
This condition is not required for superlinear convergence and is
eventually satisfied automatically due to global convergence.
\end{remark}

We present a corollary derived from Theorem~\ref{th:convgstr}, applicable to strongly quasiconvex functions and, by extension, to strongly convex functions.
We recall that every strongly quasiconvex function is $2$-supercoercive by 
Remark~\ref{rem:supercoercive}~\ref{rem:supercoercive:a}.

\begin{corollary}\label{cor:convgstr}
 Let $\gf: \R^n\to \Rinf$ be a proper, lsc, and strongly quasiconvex function with modulus $\phi=\rho t^2$ on a closed convex set $C \subseteq \dom{\gf}$.
 If the sequence $\{x^{k}\}_{k\in \Nz}$ is generated by Algorithm~\ref{alg:HiPPA} and 
$\{\ov{x}\}=\argmin{x \in C} \gf(x)$, then the following statements hold:
    \begin{enumerate}[label=(\textbf{\alph*}), font=\normalfont\bfseries, leftmargin=0.7cm]

     \item \label{cor:convgstr:b} If $p=2$, the convergence rate is linear, i.e., 
    \[
\frac{\lVert \ov{x} - x^{k+1} \rVert}{\lVert \ov{x} -x^{k}\rVert}
 \leq  \frac{1}{1+\gamma_{\bs\min}\rho}<1, \qquad \forall k\in \Nz.
\]

 \item \label{cor:convgstr:e} If  $p>2$, then the convergence rate is $Q$-superlinear of order $p-1$. In particular, for the constant
  $c=(\gamma_{\bs\min}\rho_q)^{-1}$,
    \[
    \lVert x^{k+1} - \ov{x} \rVert\leq c \lVert x^{k} - \ov{x}\rVert^{p-1}, \qquad \forall k\in \Nz.
    \]
    \end{enumerate}
\end{corollary}

\subsection{Complexity analysis}
In this subsection, we investigate the iteration and functional complexity of HiPPA. 
Beyond global convergence and asymptotic rate results, complexity estimates provide explicit bounds on the number of iterations required to reach a prescribed accuracy, either in terms of the iterate distance to the unique minimizer or the objective function value gap. Moreover, we derive a non-asymptotic bound on the number of iterations needed to satisfy a practical stopping criterion based on successive iterates. 

In Theorem~\ref{th:nonemprox}~\ref{th:nonemprox:c}, it is shown that $\argmin{x \in C} \gf(x) = \bs{\rm Fix}(\prox{\gf}{\gamma}{p}(C))$. Hence, a criterion for stopping Algorithm \ref{alg:HiPPA} is $x^{k+1} = x^k$. In practice, we replace this condition with 
$\Vert x^{k+1} - x^k \Vert \leq \varepsilon$ for a given tolerance $\varepsilon > 0$. The following theorem reveals the maximum number of iterations required to achieve the desired accuracy.

\begin{theorem}[Iteration complexity under step-size stopping criteria]\label{pro:reqexe}
Let the assumptions of Theorem~\ref{prop:propertiesppa} hold. If the sequence $\{x^k\}_{k \in \Nz}$ is generated by Algorithm~\ref{alg:HiPPA}, then at most  
\[\frac{p\gamma_{\bs\max}(\gf(x^0)-\bs\inf_{x\in C}\gf(x))}{\varepsilon^p},\]
iterations of Algorithm~\ref{alg:HiPPA} are needed to satisfy the stopping criterion $\Vert x^{k+1} - x^k \Vert \leq \varepsilon$ for a given tolerance $\varepsilon > 0$.
\end{theorem}
\begin{proof}
Summing the inequality \eqref{eq1:prop:propertiesppa} over the first $N>0$ iterations yields
\[
\sum_{k=0}^{N-1}\Vert x^{k+1} - x^k\Vert^p\leq p\gamma_{\bs\max}(\gf(x^0)-\mathop{\bs\inf}\limits_{x\in C}\gf(x)),
\]
meaning that
\[
N\bs\min_{0\leq k\leq N-1}\Vert x^{k+1} - x^k\Vert^p\leq p\gamma_{\bs\max}(\gf(x^0)-\mathop{\bs\inf}\limits_{x\in C}\gf(x)).
\]
Rearranging this inequality yields
$\frac{p\gamma_{\bs\max}(\gf(x^0)-\bs\inf_{x\in C}\gf(x))}{N}\leq \varepsilon^p$, i.e., HiPPA terminates within the number
\[
k\leq \frac{p\gamma_{\bs\max}(\gf(x^0)-\bs\inf_{x\in C}\gf(x))}{\varepsilon^p},
\]
which proves the stated iteration bound.
\end{proof}
The bound obtained in Theorem~\ref{pro:reqexe} is non-asymptotic and holds without imposing any smoothness assumptions on $\gf$.

In the following, we explore the complexity HiPPA in terms of iterations and function values.
In the following, the \textit{ceiling function}, denoted by $\lceil a \rceil$, assigns to each real number $a\in\R$ the smallest integer greater than or equal to it.

\begin{theorem}[Iteration complexity in distance to the minimizer]\label{th:complexity}
    Let $\gf: \R^n\to \Rinf$ be a proper, lsc, coercive, and uniformly quasiconvex function with modulus $\phi$ on a closed convex set $C \subseteq \dom{\gf}$.
 If the sequence $\{x^{k}\}_{k\in \Nz}$ is generated by Algorithm~\ref{alg:HiPPA} and 
$\{\ov{x}\}=\argmin{x \in C} \gf(x)$, then the following statements hold:
    \begin{enumerate}[label=(\textbf{\alph*}), font=\normalfont\bfseries, leftmargin=0.7cm]
    \item \label{th:complexity:a} 
    Let for some $\rho_q>0$, we have $\rho_q t^q\leq \phi(t)$ on $[0,1)$ for some $q\in (1,2)$, $p\in [q,2)$, and $\gamma_{\bs\min}>\frac{1}{\rho_q}$. If the sequence starts within the unit ball centered at the minimizer $\ov{x}$ and remains there, then for a given $\varepsilon\in (0,1)$, the number of iterations to reach
$\Vert \ov{x} - x^k\Vert< \varepsilon$, denoted by $\mathcal{N}(\varepsilon)$, satisfies
\begin{equation}
\label{eq:complexity-bound}
\mathcal{N}(\varepsilon)\leq \left\lceil 1+(p-1)\frac{\log(\varepsilon^{-1}) + \log\left(\Vert x^{0} - \ov{x}\Vert\right)}{\log(\rho_q\gamma_{\bs\min})}\right\rceil.
\end{equation}
In particular, the iteration complexity is
$\mathcal{O}(\log(\varepsilon^{-1}))$.

\item \label{th:complexity:b} If for some $\rho>0$, we have $\rho t^2\leq \phi(t)$ on $[0,+\infty)$     
     and $p=2$, then, for a given $\varepsilon>0$,  
     the number of iterations to reach $\Vert \ov{x} - x^k\Vert< \varepsilon$, denoted by $\mathcal{N}(\varepsilon)$, satisfies
\begin{equation}
\label{eq:complexity-bound:b}
\mathcal{N}(\varepsilon)\le \left\lceil 1+\frac{\log(\varepsilon^{-1})+\log(\Vert x^0-\ov{x}\Vert)}{\log(1 + \rho \gamma_{\bs\min})}\right\rceil,
\end{equation}
In particular, the iteration complexity is $\mathcal{O}\big(\log(\varepsilon^{-1})\big)$.

 \item \label{th:complexity:c} If $p>2$, and for some $\rho_p>0$, we have $\rho_p t^p\leq \phi(t)$ on $[0,+\infty)$, then by choosing $\gamma_{\bs\min}>\frac{1}{\rho_p}$, for a given $\varepsilon>0$,  
     the number of iterations to reach $\Vert \ov{x} - x^k\Vert< \varepsilon$, denoted by $\mathcal{N}(\varepsilon)$, satisfies
\begin{equation}
\label{eq:complexity-bound:c}
\mathcal{N}(\varepsilon)\le \left\lceil 1+\frac{\log(\varepsilon^{-1})+\log(\Vert x^0-\ov{x}\Vert)}{\frac{1}{p-1}\left(\log(p\gamma_{\bs\min}\rho_p+\widehat{\sigma}_p)-\log(p)\right)}\right\rceil,
\end{equation}
In particular, the iteration complexity is $\mathcal{O}\big(\log(\varepsilon^{-1})\big)$.

 \item \label{th:complexity:d} Let  
  for some $q\geq 2$, there exists $\rho_q>0$, such that $\rho_q t^q\leq \phi(t)$ on $[0,+\infty)$, 
 $p>q$, and $\|x^0-\ov{x}\|\le c^{-\frac{q-1}{p-q}}$. Then for a given $\varepsilon>0$, the number of iterations to reach
$\Vert \ov{x} - x^k\Vert< \varepsilon$, denoted by $\mathcal{N}(\varepsilon)$, satisfies
\[
\mathcal{N}(\varepsilon)\le \left\lceil 1+
\frac{\log\Big(\log \big(c^{-\frac{1}{\alpha-1}}\varepsilon^{-1}\big)\Big)
-\log\Big(\log(u_0^{-1})\Big)}
{\log(\alpha)}\right\rceil,
\]
where $c:=(\gamma_{\bs\min}\rho_q)^{\frac{-1}{q-1}}$, $\alpha:=\frac{p-1}{q-1}$, and 
$u_0=c^{\frac{1}{\alpha-1}}\|x^{0}-\ov{x}\|$. In particular,
\begin{equation}\label{eq:th:complexity:d:1}
\log\Big(\log \big(c^{-\frac{1}{\alpha-1}}\varepsilon^{-1}\big)\Big)
\geq \log\Big(\log(u_0^{-1})\Big),
\end{equation}
and
the iteration complexity is $\mathcal{O}\big(\log(\log(\varepsilon^{-1}))\big)$.
    \end{enumerate}    
\end{theorem}

\begin{proof}
\ref{th:complexity:a} 
If $\Vert \ov{x} - x^0\Vert< \varepsilon$, the stopping criterion is already satisfied and the algorithm terminates at the initial iterate. Therefore, in the subsequent analysis we assume that
 $\Vert \ov{x} - x^0\Vert\geq \varepsilon$, which ensures that condition \eqref{eq:th:complexity:d:1} holds.
From the proof of Theorem~\ref{th:convgstr}\,\ref{th:convgstr:a}, for all $k\in\mathbb{N}$,
\[
\Vert x^{k+1} - \ov{x}\Vert \leq \mu\Vert x^k-\ov{x}\Vert,
\]
where $\mu= (\rho_q\gamma_{\bs\min})^{-\frac{1}{p-1}}<1$.
By induction, this inequality implies
\begin{equation}\label{eq1:th:complexity:funval}
\Vert x^{k} - \ov{x}\Vert \leq \mu^{k}\Vert x^{0} - \ov{x}\Vert, \qquad \forall\, k\geq 0.
\end{equation}
Let $\mathcal{K}\in \Nz$ be the first natural number that 
$ \mu^{\mathcal{K}}\Vert x^{0} - \ov{x}\Vert< \varepsilon$, which in turn implies $\Vert x^{\mathcal{K}} - \ov{x}\Vert <\varepsilon$.
We have $\mu^{\mathcal{K}-1}\Vert x^{0} - \ov{x}\Vert\geq \varepsilon$. Hence,
\[
\log(\mu)(\mathcal{K}-1)\geq \log(\varepsilon) - \log\left(\Vert x^{0} - \ov{x}\Vert\right).
\]
Hence,
\[
\mathcal{K}\leq 1+\frac{\log(\varepsilon^{-1}) + \log\left(\Vert x^{0} - \ov{x}\Vert\right)}{\log(1/\mu)}.
\]
Since
\begin{equation}\label{eq3:th:complexity:funval}
\log(1/\mu)=\log \left((\rho_q\gamma_{\bs\min})^{\frac{1}{p-1}}\right)
=\frac{1}{p-1}\log(\rho_q\gamma_{\bs\min}),
\end{equation}
the bound~\eqref{eq:complexity-bound} follows. 
\\
\ref{th:complexity:b} From Theorem~\ref{th:convgstr}~\ref{th:convgstr:b},
$\Vert x^{k}-\ov{x}\Vert\leq \mu^k\Vert x^0-\ov{x}\Vert$ with $\mu= \frac{1}{1+\rho \gamma_{\bs\min}}$.
Consequently, similar to the proof of Assertion~\ref{th:complexity:a}, the number of iterations to reach
$\Vert \ov{x} - x^k\Vert< \varepsilon$, satisfies \eqref{eq:complexity-bound:b},
which shows $\lVert x^{k} - \ov{x}\rVert\le\varepsilon$ is guaranteed within  $\mathcal{O}\big(\log(\varepsilon^{-1})\big)$. 
\\
\ref{th:complexity:c}
From Theorem~\ref{th:convgstr}~\ref{th:convgstr:d}, 
$\Vert x^{k}-\ov{x}\Vert\leq \mu^k\Vert x^0-\ov{x}\Vert$, where $\mu= \left(\frac{p}{p\gamma_{\bs\min}\rho_p+\widehat{\sigma}_p}\right)^{\frac{1}{p-1}}$.
Hence, similar to the proof of Assertion~\ref{th:complexity:a}, the number of iterations to reach
$\Vert \ov{x} - x^k\Vert< \varepsilon$, satisfies \eqref{eq:complexity-bound:c},
indicating $\lVert x^{k} - \ov{x}\rVert\le\varepsilon$ is guaranteed within $\mathcal{O}\big(\log(\varepsilon^{-1})\big)$ number of iterations. 
\\
\ref{th:complexity:d} Set $e_k:=\|x^k-\ov{x}\|$ for $k\in\Nz$. By Theorem~\ref{th:convgstr}~\ref{th:convgstr:e}, we have
\begin{equation}\label{eq:super-rec}
e_{k+1}\le c\, e_k^{\alpha},\qquad \forall k\in\Nz,
\end{equation}
where $c:=(\gamma_{\bs\min}\rho_q)^{\frac{-1}{q-1}}$ and $\alpha:=\frac{p-1}{q-1}>1$. 
From the assumption,
\begin{equation}\label{eq:k0-choice}
e_{0}< c^{-\frac{1}{\alpha-1}}.
\end{equation}
Define $u_k:=c^{\frac{1}{\alpha-1}}\,e_k$. Then \eqref{eq:super-rec} implies, for all $k\in\Nz$,
\[
u_{k+1} = c^{\frac{1}{\alpha-1}}e_{k+1} \le
c^{\frac{1}{\alpha-1}}\cdot c\, e_k^{\alpha}
= \big(c^{\frac{1}{\alpha-1}}e_k\big)^{\alpha} = u_k^{\alpha}.
\]
In particular, by \eqref{eq:k0-choice} we have $u_{0}< 1$, and therefore the sequence
$\{u_k\}_{k\in\Nz}$ is nonincreasing. Iterating $u_{k+1}\le u_k^{\alpha}$ yields
\begin{equation}\label{eq:uk-bound}
u_{k}\le u_{0}^{\alpha^{k}},\qquad \forall k\in\Nz.
\end{equation}
Let $\mathcal{K}\in \Nz$ be the first natural number that 
$u_{0}^{\alpha^{\mathcal{K}}} < c^{\frac{1}{\alpha-1}}\varepsilon$, which in turn implies
$e_{\mathcal{K}}< \varepsilon$. We have $u_{0}^{\alpha^{\mathcal{K}-1}} \geq c^{\frac{1}{\alpha-1}}\varepsilon$. Hence,
\[
\alpha^{\mathcal{K}-1} \leq \frac{\log\big(c^{-\frac{1}{\alpha-1}}\varepsilon^{-1}\big)}
{\log(u_{0}^{-1})}.
\]
Taking logarithms and using $\alpha>1$ shows that $e_{\mathcal{K}}< \varepsilon$ holds whenever
\[
\mathcal{K}\ \leq 1+
\frac{\log\Big(\log \big(c^{-\frac{1}{\alpha-1}}\varepsilon^{-1}\big)\Big)
-\log\big(\log(u_{0}^{-1})\big)}
{\log(\alpha)}.
\]
\end{proof}

\begin{remark}[On local complexity and the entry phase]
The initialization assumptions in cases~\ref{th:complexity:a} and~\ref{th:complexity:d} of Theorem~\ref{th:complexity} are imposed only to state clean complexity bounds from the initial index $k=0$. If these assumptions are not imposed, the global convergence result in Theorem~\ref{prop:propertiesppa} still implies that the required local conditions hold after some finite iterations $k_0\geq 0$. In case~\ref{th:complexity:a}, the estimate applies from $x^{k_0}$ once $\|x^k-x^*\|<1$ for all $k\ge k_0$, with an additional entry cost $k_0$. Similarly, in the superlinear case~\ref{th:complexity:d}, the estimate $\mathcal{O}(\log(\log(\varepsilon^{-1})))$ applies after the iterates enter the region $\|x^{k_0}-x^*\|<c^{-1/(\alpha-1)}$, where $c:=(\gamma_{\bs\min}\rho_q)^{-1/(q-1)}$ and $\alpha:=(p-1)/(q-1)$. Since the present analysis does not provide an explicit problem-independent upper bound for $k_0$, this estimate should be understood as a local (asymptotic) complexity bound in the superlinear regime rather than as a directly verifiable global bound from an arbitrary starting point. The same interpretation applies to the function-value complexity estimates in Theorem~\ref{th:complexity2} below.
\end{remark}

\begin{theorem}[Function-value complexity of HiPPA]\label{th:complexity2}
    Let $\gf: \R^n\to \Rinf$ be a proper, lsc, coercive, and uniformly quasiconvex function with modulus $\phi$ on a closed convex set $C \subseteq \dom{\gf}$.
 If the sequence $\{x^{k}\}_{k\in \Nz}$ is generated by Algorithm~\ref{alg:HiPPA} and 
$\{\ov{x}\}=\argmin{x \in C} \gf(x)$, then the following statements hold:
    \begin{enumerate}[label=(\textbf{\alph*}), font=\normalfont\bfseries, leftmargin=0.7cm]
    \item \label{th:complexity:a2} 
     Let for some $\rho_q>0$, we have $\rho_q t^q\leq \phi(t)$ on $[0,1)$ for some $q\in (1,2)$, $p\in [q,2)$, and $\gamma_{\bs\min}>\frac{1}{\rho_q}$. If the sequence starts within the unit ball centered at the minimizer $\ov{x}$ and remains there, then for a given $\varepsilon\in (0,1)$, the number of iterations to reach
$\gf (x^{k}) -\gf (\ov{x})< \varepsilon$, denoted by $\mathcal{N}^{\gf}(\varepsilon)$, satisfies
 \[
 \mathcal{N}^{\gf}(\varepsilon)\leq \left\lceil 2 + \frac{\log(\varepsilon^{-1})+\log\left(\frac{\Vert x^{0} - \ov{x}\Vert^p}{p\gamma_{\bs\min}}\right)}{p\log(1/\mu)}\right\rceil.
 \]
In particular, the local iteration complexity 
is $\mathcal{O}\big(\log(\varepsilon^{-1})\big)$.

\item \label{th:complexity:b2} If for some $\rho>0$, we have $\rho t^2\leq \phi(t)$ on $[0,+\infty)$     
     and $p=2$, then, for a given $\varepsilon>0$,  
the number of iterations to reach
$\gf (x^{k}) -\gf (\ov{x})< \varepsilon$, denoted by $\mathcal{N}^{\gf}(\varepsilon)$, satisfies
\begin{equation}\label{th:complexity:b2:eq1}
    \mathcal{N}^{\gf}(\varepsilon)\leq \left\lceil
    2+ \frac{\log(\varepsilon^{-1})+\log\left(\frac{\Vert x^{0} - \ov{x}\Vert^p}{p\gamma_{\bs\min}}\right)}{p\log(1+\rho\gamma_{\bs\min})}\right\rceil.
\end{equation}
In particular, the iteration complexity in function value is
$k = \mathcal{O}\big(\log(\varepsilon^{-1})\big)$.

 \item \label{th:complexity:c2} If $p>2$, and for some $\rho_p>0$, we have $\rho_p t^p\leq \phi(t)$ on $[0,+\infty)$, then by choosing $\gamma_{\bs\min}>\frac{1}{\rho_p}$, for a given $\varepsilon>0$,  
  the number of iterations to reach
$\gf (x^{k}) -\gf (\ov{x})< \varepsilon$, denoted by $\mathcal{N}^{\gf}(\varepsilon)$, satisfies
\begin{equation}\label{th:complexity:b2:eq2}
    \mathcal{N}^{\gf}(\varepsilon)\leq \left\lceil
    2+\frac{\log(\varepsilon^{-1})+\log\left(\frac{\Vert x^{0} - \ov{x}\Vert^p}{p\gamma_{\bs\min}}\right)}{\frac{1}{p-1}\left(\log(p\gamma_{\bs\min}\rho_p+\widehat{\sigma}_p)-\log(p)\right)}\right\rceil.
\end{equation}
In particular, the iteration complexity in function value is
$k = \mathcal{O}\big(\log(\varepsilon^{-1})\big)$.

 \item \label{th:complexity:d2} 
 Let for some $q\geq 2$, there exists $\rho_q>0$, such that $\rho_q t^q\leq \phi(t)$ on $[0,+\infty)$, 
 $p>q$, and $\|x^0-\ov{x}\|\le c^{-\frac{q-1}{p-q}}$. Then for a given $\varepsilon>0$, 
  the number of iterations to reach
$\gf (x^{k}) -\gf (\ov{x})< \varepsilon$, denoted by $\mathcal{N}^{\gf}(\varepsilon)$, satisfies
\begin{equation}\label{th:complexity:b2:eq3}
\mathcal{N}^{\gf}(\varepsilon) \leq \left\lceil 2+
\frac{\log\Big(\log \big(c^{-\frac{p}{\alpha-1}}\left(p \gamma_{\bs\min}\right)^{-1}\varepsilon^{-1}\big)\Big)
-\log\big(p\log(u_0^{-1})\big)}
{\log(\alpha)}\right\rceil,
\end{equation}
where $c:=(\gamma_{\bs\min}\rho_q)^{\frac{-1}{q-1}}$, $\alpha:=\frac{p-1}{q-1}$, and $u_0=c^{\frac{1}{\alpha-1}}\|x^{0}-\ov{x}\|$.
In particular,
\begin{equation}\label{eq:th:complexity:d2:1}
\log\Big(\log \big(c^{-\frac{p}{\alpha-1}}\left(p \gamma_{\bs\min}\right)^{-1}\varepsilon^{-1}\big)\Big)
\geq\log\big(p\log(u_0^{-1})\big)
\end{equation}
and
the iteration complexity in function value is
$\mathcal{O}\big(\log(\log(\varepsilon^{-1}))\big)$.
    \end{enumerate}

\end{theorem}
\begin{proof}
\ref{th:complexity:a2} 
From the definition, for each $p>1$,
\[
 \gf (x^{k+1}) + \frac{1}{p \gamma_{k}} \lVert x^{k+1} - x^{k} \rVert^{p}\leq \gf (\ov{x}) 
 + \frac{1}{p \gamma_{k}} \lVert \ov{x} - x^{k} \rVert^{p},
\]
which yields
\begin{equation}\label{eq2:th:complexity:funval}
\gf (x^{k+1}) -\gf (\ov{x})\leq  \frac{1}{p \gamma_{k}} \lVert \ov{x} - x^{k} \rVert^{p}
\leq  \frac{1}{p \gamma_{\bs\min}} \lVert \ov{x} - x^{k} \rVert^{p}.
\end{equation}
Together \eqref{eq1:th:complexity:funval} and setting $\mu:= (\rho_q\gamma_{\min})^{-\frac{1}{p-1}}<1$, this implies
\[
\gf (x^{k+1}) -\gf (\ov{x})\leq  \frac{1}{p \gamma_{\bs\min}} \lVert \ov{x} - x^{k} \rVert^{p}\leq  \frac{1}{p \gamma_{\bs\min}} \mu^{pk}\Vert x^{0} - \ov{x}\Vert^p.
\]
Let $\mathcal{K}\in \Nz$ be the first natural number that 
$\frac{1}{p \gamma_{\bs\min}} \mu^{p\mathcal{K}}\Vert x^{0} - \ov{x}\Vert^p<\varepsilon$, which in turn implies
$\gf (x^{\mathcal{K}+1}) -\gf (\ov{x})< \varepsilon$. We have 
$\frac{1}{p \gamma_{\bs\min}} \mu^{p(\mathcal{K}-1)}\Vert x^{0} - \ov{x}\Vert^p\geq\varepsilon$. Hence,
\[
\mathcal{K} -1\leq \frac{\log(\varepsilon^{-1})+\log\left(\frac{\Vert x^{0} - \ov{x}\Vert^p}{p\gamma_{\bs\min}}\right)}{p\log(1/\mu)}.
\]
Using \eqref{eq3:th:complexity:funval},
\[
\mathcal{K} \leq 1+ \frac{p-1}{p}\frac{\log(\varepsilon^{-1})+\log\left(\frac{\Vert x^{0} - \ov{x}\Vert^p}{p\gamma_{\bs\min}}\right)}{\log(\rho_q\gamma_{\bs\min})}.
\]
\ref{th:complexity:b} Similar to the proof of Assertion~\ref{th:complexity:a} and from the proof of Theorem~\ref{th:complexity}~\ref{th:complexity:b}, 
\[
\gf (x^{k+1}) -\gf (\ov{x})\leq  \frac{1}{p \gamma_{\bs\min}} \lVert \ov{x} - x^{k} \rVert^{p}
\leq  \frac{1}{p \gamma_{\bs\min}} \mu^{pk}\Vert x^{0} - \ov{x}\Vert^p,
\]
where $\mu= \frac{1}{1+\rho \gamma_{\bs\min}}$.
Hence,  the number of iterations to reach
$\gf (x^{k}) -\gf (\ov{x})< \varepsilon$, satisfies \eqref{th:complexity:b2:eq1}.
\\
\ref{th:complexity:c2}
Similar to the proof of Assertion~\ref{th:complexity:a} and from the proof of Theorem~\ref{th:complexity}~\ref{th:complexity:c}, 
\[
\gf (x^{k+1}) -\gf (\ov{x})\leq  \frac{1}{p \gamma_{\bs\min}} \lVert \ov{x} - x^{k} \rVert^{p}
\leq  \frac{1}{p \gamma_{\bs\min}} \mu^{pk}\Vert x^{0} - \ov{x}\Vert^p,
\]
where $\mu= \left(\frac{p}{p\gamma_{\bs\min}\rho_p+\widehat{\sigma}_p}\right)^{\frac{1}{p-1}}$.
Hence,  the number of iterations to reach
$\gf (x^{k}) -\gf (\ov{x})< \varepsilon$, satisfies \eqref{th:complexity:b2:eq2}.
\\
\ref{th:complexity:d2}
If $\frac{1}{p \gamma_{\bs\min}}\|x^0-\ov{x}\|^{p}<\varepsilon$, the result is evident. Therefore, in the subsequent analysis we assume that
$\frac{1}{p \gamma_{\bs\min}}\|x^0-\ov{x}\|^{p}\geq\varepsilon$, which ensures that condition \eqref{eq:th:complexity:d2:1} holds.
Let $e_k:=\|x^k-\ov{x}\|$. Similar to the proof of Theorem~\ref{th:complexity}~\ref{th:complexity:d}, by defining
 $u_k:=c^{\frac{1}{\alpha-1}}\,e_k$, where $c:=(\gamma_{\bs\min}\rho_q)^{\frac{-1}{q-1}}$ and $\alpha:=\frac{p-1}{q-1}>1$, 
 \[
u_{k}\le u_{0}^{\alpha^{k}},\qquad \forall k\in\Nz.
 \]
This implies
\[
e_{k}\leq c^{-\frac{1}{\alpha-1}}\left(c^{\frac{1}{\alpha-1}}\,e_{0}\right)^{\alpha^{k}},\qquad \forall k\in\Nz.
\]
Together with \eqref{eq2:th:complexity:funval}, this implies
\[
\gf (x^{k+1}) -\gf (\ov{x})
\leq  \frac{1}{p \gamma_{\bs\min}} e_{k}^{p}\leq \frac{1}{p \gamma_{\bs\min}}c^{-\frac{p}{\alpha-1}}\left(c^{\frac{1}{\alpha-1}}\,e_{0}\right)^{p\alpha^{k}},\qquad \forall k\in\Nz.
\]
Let $\mathcal{K}\in \Nz$ be the first natural number that 
\[
\frac{1}{p \gamma_{\bs\min}}c^{-\frac{p}{\alpha-1}}\left(c^{\frac{1}{\alpha-1}}\,e_{0}\right)^{p\alpha^{\mathcal{K}}}<\varepsilon.
\]
Then, utilizing \eqref{eq:k0-choice},
\[
\mathcal{K}\ \leq 1+
\frac{\log\Big(\log \big(c^{-\frac{p}{\alpha-1}}\left(p \gamma_{\bs\min}\right)^{-1}\varepsilon^{-1}\big)\Big)
-\log\big(p\log(u_{0}^{-1})\big)}
{\log(\alpha)}.
\]
Hence,  the number of iterations to reach
$\gf (x^{k}) -\gf (\ov{x})< \varepsilon$, satisfies \eqref{th:complexity:b2:eq3}.
\end{proof}

\begin{remark}
The bound in Theorem~\ref{pro:reqexe} has a different role from the sharper bounds derived in Theorems~\ref{th:complexity}~and~\ref{th:complexity2}.  It is a non-asymptotic guarantee for the computable stopping criterion $\|x^{k+1}-x^k\|\le \varepsilon$, and follows only from the descent estimate and the boundedness of the proximal parameters.   Theorems~\ref{th:complexity}~and~\ref{th:complexity2}, on the other hand, use the additional assumptions of the form $\gf(t)\ge \rho_q t^q$ to convert the linear and superlinear estimates of Theorem~\ref{th:convgstr} into distance and function-value complexity bounds by inverting the corresponding recursive estimates. Thus, Theorem~\ref{pro:reqexe} should be viewed as a general worst-case step-size stopping guarantee.
\end{remark}


\section{Conclusion}\label{sec:conclusion}

This paper investigated the asymptotic convergence properties of the high-order proximal-point algorithm (HiPPA) with two primary objectives: (i) establishing convergence to a global minimizer, and (ii) achieving faster convergence rates beyond the classical sublinear bounds. To this end, we introduced and analyzed the class of uniformly quasiconvex optimization problems, exploring their key structural properties—including characterization, calculus rules, coercivity, and the exclusion of local minima and saddle points. We demonstrated that the sequence generated by HiPPA converges to a stationary point of the underlying uniformly quasiconvex function, which, by construction, is also its unique global minimizer. To address the second objective, we further assumed that the modulus function of uniform quasiconvexity satisfies the inequality $\phi(t) \geq \rho_q t^q$ over an interval $\mathcal{I}$. Under this condition, we established the following convergence rates:
\begin{enumerate}[label=(\textbf{\alph*}), font=\normalfont\bfseries, leftmargin=0.7cm]
    \item[$\bullet$] For $q \in (1,2)$ and $\mathcal{I} = [0,1)$, HiPPA achieves a local linear rate for $p \in [q,2)$;
    \item[$\bullet$] For $q = 2$ and $\mathcal{I} = [0,\infty)$, HiPPA converges linearly when $p = 2$;
    \item[$\bullet$] For $p = q > 2$ and $\mathcal{I} = [0,\infty)$, HiPPA again achieves linear convergence;
    \item[$\bullet$] For $q \geq 2$ and $\mathcal{I} = [0,\infty)$, HiPPA attains a superlinear rate when $p > q$.
\end{enumerate}
To the best of our knowledge, this work presented the first comprehensive convergence rate analysis of HiPPA. Some of the results obtained are new even for the special cases of strongly or uniformly convex functions, thus providing novel insights into the optimization of generalized convex problems.

\bibliographystyle{siamplain}
\bibliography{references.bib}

@article{Ahookhosh24,
  title={High-order methods beyond the classical complexity bounds: Inexact high-order proximal-point methods},
  author={Ahookhosh, Masoud and Nesterov, Yurii},
  journal={Mathematical Programming},
  volume = {208},
  pages={365--407},
  year={2024},
  publisher={Springer}
}

@article{Ahookhosh23,
author = {Ahookhosh, Masoud and Nesterov, Yurii},
title = {High-order methods beyond the classical complexity bounds: Inexact high-order proximal-point methods with segment search},
journal = {Submitted manuscript},
year = {2023},
}

@article{Ahookhosh21,
author = {Ahookhosh, Masoud and Themelis, Andreas and Patrinos, Panagiotis},
title = {A {B}regman forward-backward linesearch algorithm for nonconvex composite optimization: Superlinear convergence to nonisolated local minima},
journal = {SIAM Journal on Optimization},
volume = {31},
pages = {653--685},
year = {2021}
}

@article{Asi2019Stochastic,
author = {Asi, Hilal and Duchi, John C.},
title = {Stochastic (Approximate) Proximal Point Methods: Convergence, Optimality, and Adaptivity},
journal = {SIAM Journal on Optimization},
volume = {29},
number = {3},
pages = {2257-2290},
year = {2019}
}

@article{aujol2024fista,
  title={{FISTA} is an automatic geometrically optimized algorithm for strongly convex functions},
  author={Aujol, J-F and Dossal, Ch and Rondepierre, Aude},
  journal={Mathematical Programming},
  volume={204},
  number={1},
  pages={449--491},
  year={2024},
  publisher={Springer}
}

@book{Bauschke17,
  title={Convex Analysis and Monotone Operator Theory in Hilbert Spaces},
  edition={2},
  author={Bauschke, H.-H. and Combettes, P.-L.},
  year={2017},
  publisher={Springer Cham}
}

@article{flores2021characterizing,
  title={Characterizing quasiconvexity of the pointwise infimum of a family of arbitrary translations of quasiconvex functions, with applications to sums and quasiconvex optimization},
  author={Flores-Baz{\'a}n, F and Garc{\'\i}a, Yboon and Hadjisavvas, Nicolas},
  journal={Mathematical Programming},
  volume={189},
  number={1},
  pages={315--337},
  year={2021},
  publisher={Springer}
}

@book{boyd2004convex,
  title={Convex optimization},
  author={Boyd, Stephen P and Vandenberghe, Lieven},
  year={2004},
  publisher={Cambridge university press}
}

@article{Bredies2022Degenerate,
author = {Bredies, Kristian and Chenchene, Enis and Lorenz, Dirk A. and Naldi, Emanuele},
title = {Degenerate Preconditioned Proximal Point Algorithms},
journal = {SIAM Journal on Optimization},
volume = {32},
number = {3},
pages = {2376-2401},
year = {2022}
}

@article{Cambini2003coercivity,
  title={Coercivity concepts and recession function in constrained problems},
  author={Cambini, Riccardo and Carosi, Laura},
  journal={International Journal of Mathematical Sciences},
  volume={2},
  pages={83--96},
  year={2003}
}

@book{cambini2008generalized,
  title={Generalized Convexity and Optimization: Theory and Applications},
  author={Cambini, Alberto and Martein, Laura},
  volume={616},
  year={2009},
  publisher={Springer Science \& Business Media}
}

@article{Cartis26second,
author = {Cartis, C. and Zhu, W.},
title = {Second-order methods for quartically-regularised cubic polynomials, with applications to high-order tensor methods},
journal = { Mathematical Programming},
volume = {215},
pages = {669-715},
year = {2026}
}

@article{Cartis25Cubic,
author = {Cartis, C. and Zhu, W.},
title = {Cubic-quartic regularization models for solving polynomial subproblems in third-order tensor methods},
journal = {Mathematical Programming},
year = {2025}
}

@book{clarke2013functional,
  title={Functional Analysis, Calculus of Variations and Optimal Control},
  author={Clarke, Francis},
  year={2013},
  publisher={Springer}
}

@incollection{crouzeix2005continuity,
  author      = "Crouzeix, Jean-Pierre",
  title       = "Continuity and differentiability of quasiconvex functions",
  editor      = "Hadjisavvas, N. and Komlósi, S. and Schaible, S.",
  booktitle   = "Handbook of Generalized Convexity and Generalized Monotonicity",
  publisher   = "Springer",
  year        = 2005,
  pages       = "121--149",
}

@article{davis2022proximal,
  title={Proximal methods avoid active strict saddles of weakly convex functions},
  author={Davis, Damek and Drusvyatskiy, Dmitriy},
  journal={Foundations of Computational Mathematics},
  volume={22},
  number={2},
  pages={561--606},
  year={2022},
  publisher={Springer}
}

@article{Davis2022Escaping,
author = {Davis, Damek and D\'{\i}az, Mateo and Drusvyatskiy, Dmitriy},
title = {Escaping Strict Saddle Points of the {M}oreau Envelope in Nonsmooth Optimization},
journal = {SIAM Journal on Optimization},
volume = {32},
number = {3},
pages = {1958-1983},
year = {2022}
}

@book{dhara2011optimality,
  title={Optimality Conditions in Convex Optimization: A Finite-Dimensional View},
  author={Dhara, Anulekha and Dutta, Joydeep},
  year={2011},
  publisher={CRC Press}
}

@article{GLM-survey,
  title={Strongly quasiconvex functions: What we know (so far)},
  author={Grad, S.-M. and Lara, F. and Marcavillaca, R.T.},
  journal={Journal of Optimization Theory and Applications},
  volume={205:38},
  year={2025}
}

@article{grimmer2025some,
  title={Some primal-dual theory for subgradient methods for strongly convex optimization},
  author={Grimmer, Benjamin and Li, Danlin},
  journal={Mathematical Programming},
  volume = {214},
  pages={759--788},
  year={2025},
  publisher={Springer}
}

@article{Gu2020Tight,
author = {Gu, Guoyong and Yang, Junfeng},
title = {Tight Sublinear Convergence Rate of the Proximal Point Algorithm for Maximal Monotone Inclusion Problems},
journal = {SIAM Journal on Optimization},
volume = {30},
number = {3},
pages = {1905-1921},
year = {2020}
}

@article{Guler9291Convergence,
author = {G\"{u}ler, Osman},
title = {On the Convergence of the Proximal Point Algorithm for Convex Minimization},
journal = {SIAM Journal on Control and Optimization},
volume = {29},
number = {2},
pages = {403-419},
year = {1991}
}

@article{Guler92,
  author = {G\"{u}ler, Osman},
  title = {New Proximal Point Algorithms for Convex Minimization},
  journal = {SIAM Journal on Optimization},
  volume = {2},
  pages = {649--664},
  year = {1992}
}

@book{hiriart1996convex,
  title={Convex Analysis and Minimization Algorithms I: Fundamentals},
  author={Hiriart-Urruty, Jean-Baptiste and Lemar{\'e}chal, Claude},
  volume={305},
  year={1996},
  publisher={Springer science \& business media}
}

@article{iusem2018second,
  title={Second order asymptotic functions and applications to quadratic programming},
  author={Iusem, A and Lara, Felipe},
  journal={Journal of Convex Analysis},
  volume={25},
  pages={271--291},
  year={2018}
}

@article{Iusem2022proximal,
  title={Proximal point algorithms for quasiconvex pseudomonotone equilibrium problems},
  author={Iusem, A and Lara, Felipe},
  journal={Journal of Optimization Theory and Applications},
  volume={193},
  number={1},
  pages={443--461},
  year={2022},
  publisher={Springer}
}

@article{Iusem20204two-step,
  author ={Iusem, A. and  Lara, F. and Marcavillaca, R.T. and Yen, L.H.},
  title = {A Two-Step Proximal Point Algorithm for Nonconvex Equilibrium Problems with Applications to Fractional Programming},
  journal = {Journal of Global Optimization},
  volume = {90},
  pages = {755--779},
  year = {2024}
}

@article{Josz2023Certifying,
author = {Josz, C\'{e}dric and Li, Xiaopeng},
title = {Certifying the Absence of Spurious Local Minima at Infinity},
journal = {SIAM Journal on Optimization},
volume = {33},
number = {3},
pages = {1416-1439},
year = {2023}
}

@article{jovanovivc1996note,
  title={A note on strongly convex and quasiconvex functions},
  author={Jovanovi{\v{c}}, MV},
  journal={Mathematical Notes},
  volume={60},
  pages={584--585},
  year={1996},
  publisher={Springer}
}

@article{Kabgani24itsopt,
  author = {Kabgani, A. and Ahookhosh, M.},
  title = {Its{OPT}: An inexact two-level smoothing framework for nonconvex optimization via high-order Moreau envelope},
  year = {2026},
  journal = {SIAM Journal on Optimization (To appear)}
}

@article{Kabgani25itsdeal,
  author = {Kabgani, A. and Ahookhosh, M.},
  title = {Its{DEAL}: Inexact two-level smoothing descent algorithms for weakly convex optimization},
  year = {2025},
  url = {https://doi.org/10.48550/arXiv.2501.02155},
  journal = {arXiv},
  archiveprefix = {2501.02155}
}

@article{Kabgani25fundamental,
  author = {Kabgani, A. and Ahookhosh, M.},
  title = {On Fundamental Properties of High-Order Forward-Backward Envelope},
  year = {2026},
  volume={210, 14},
  journal={Journal of Optimization Theory and Applications}
}

@article{Kabgani25First,
  author = {Kabgani, A. and Ahookhosh, M.},
  title = {First-order majorization-minimization meets high-order majorant: Boosted inexact high-order forward-backward method},
  year = {2025},
  url = {https://doi.org/10.48550/arXiv.2510.22231},
  journal = {arXiv},
  archiveprefix = {2510.22231}
}

@article{Kabgani2023,
  title={Semistrictly and neatly quasiconvex programming using lower global subdifferentials},
  author={Kabgani, A. and Lara, F.},
  journal={Journal of Global Optimization},
  volume={86},
  pages={845-865},
  year={2023}
}

@article{Kabganidiff,
  author = {Kabgani, A. and Ahookhosh, M.},
  title = {Moreau Envelope and Proximal-Point Methods Under the Lens of High-Order Regularization},
  year = {2025},
  volume={33, 47},
  journal={Set-Valued and Variational Analysis}
}

@article{Kab-Lara-2,
  title={Strong subdifferentials: Theory and 
  applications in nonconvex optimization},
  author={Kabgani, A. and Lara, F.},
  journal={ Journal of Global Optimization},
  volume={84},
  pages={349-368},
  year={2022}
}

@article{KecisThibault15,
title = {Moreau envelopes of $s$-lower regular functions},
journal = {Nonlinear Analysis: Theory, Methods \& Applications},
volume = {127},
pages = {157-181},
year = {2015},
author = {Kecis, I. and Thibault, L.}
}

@article{Lara2022strongly,
  title={On strongly quasiconvex functions: Existence results and proximal point algorithms},
  author={Lara, Felipe},
  journal={Journal of Optimization Theory and Applications},
  volume={192},
  number={3},
  pages={891--911},
  year={2022},
  publisher={Springer}
}

@article{lee2019first,
  title={First-order methods almost always avoid strict saddle points},
  author={Lee, Jason D and Panageas, Ioannis and Piliouras, Georgios and Simchowitz, Max and Jordan, Michael I and Recht, Benjamin},
  journal={Mathematical programming},
  volume={176},
  pages={311--337},
  year={2019},
  publisher={Springer}
}

@article{Luque19849Asymptotic,
author = {Luque, Fernando Javier},
title = {Asymptotic Convergence Analysis of the Proximal Point Algorithm},
journal = {SIAM Journal on Control and Optimization},
volume = {22},
number = {2},
pages = {277-293},
year = {1984}
}

@article{ma2022blessing,
  title={Blessing of depth in linear regression: Deeper models have flatter landscape around the true solution},
  author={Ma, Jianhao and Fattahi, Salar},
  journal={Advances in Neural Information Processing Systems},
  volume={35},
  pages={34334--34346},
  year={2022}
}

@article{martinet1970breve,
  author = {Martinet, Bernard},
  title = {R\'{e}gularisation d'in\'{e}quations variationnelles par approximations successives},
  journal = {Revue Francaise d'informatique et de Recherche operationelle},
  volume = {4},
  pages = {154--158},
  year = {1970}
}

@article{martinet1972determination,
  author = {Martinet, Bernard},
  title = {D\'{e}termination approch\'{e}e d'un point fixe d'une application pseudo-contractante},
  journal = {Cas de l'application prox,”Comptes Rendus de l'Academie des Sciences, Paris},
  volume = {274},
  pages = {163--165},
  year = {1972}
}

@article{McRae2024Benign,
author = {McRae, Andrew D. and Boumal, Nicolas},
title = {Benign Landscapes of Low-Dimensional Relaxations for Orthogonal Synchronization on General Graphs},
journal = {SIAM Journal on Optimization},
volume = {34},
number = {2},
pages = {1427-1454},
year = {2024}
}

@article{Milzarek2024Semismooth,
author = {Milzarek, Andre and Schaipp, Fabian and Ulbrich, Michael},
title = {A Semismooth {N}ewton Stochastic Proximal Point Algorithm with Variance Reduction},
journal = {SIAM Journal on Optimization},
volume = {34},
number = {1},
pages = {1157-1185},
year = {2024}
}

@book{Mordukhovich2018,
  title={Variational Analysis and Applications},
  author={Mordukhovich, Boris Sholimovich},
  volume={},
  year={2018},
  publisher={Springer Cham}
}

@article{Moreau65,
  title={Proximit{\'e} et dualit{\'e} dans un espace {H}ilbertien},
  author={Moreau, Jean-Jacques},
  journal={Bulletin de la Soci{\'e}t{\'e} Math{\'e}matique de France},
  volume={93},
  pages={273-299},
  year={1965},
  publisher={Soci{\'e}t{\'e} Math{\'e}matique de France}
}

@article{nam2024strong,
  title={On strong quasiconvexity of functions in infinite dimensions},
  author={Nam, Nguyen Mau and Sharkansky, Jacob},
  journal={Optimization Letters},
  volume={20},
  pages={815--840},
  year={2025}
}

@book{Nesterov2018,
  title={Lectures on Convex Optimization},
  author={Nesterov, Yurii},
  year={2018},
  edition ={Second},
  publisher={Springer Cham}
}

@article{Nesterov2023a,
  author = {Nesterov, Yurii},
  title = {Inexact Accelerated High-Order Proximal-Point Methods},
  journal = {Mathematical Programming},
  volume = {197},
  pages = {1--26},
  year = {2023}
}

@article{nesterov2021inexact,
  title={Inexact high-order proximal-point methods with auxiliary search procedure},
  author={Nesterov, Yurii},
  journal={SIAM Journal on Optimization},
  volume={31},
  pages={2807--2828},
  year={2021},
  publisher={SIAM}
}

@article{Parikh14,
  author = {Parikh, Neal and Boyd, Stephen},
  title = {Proximal Algorithms},
  journal = {Foundations and Trends\textsuperscript{\textregistered} in Optimization},
  volume = {1},
  pages = {127--239},
  year = {2014}
}

@article{Poliquin96,
  title={Prox-regular functions in variational analysis},
  author={Poliquin, R-A. and Rockafellar, R-T.},
  journal={Transactions of the American Mathematical Society},
  volume={348},
  pages={1805-1838},
  year={1996}
}

@article{polyak1966existence,
  title={Existence theorems and convergence of minimizing sequences in extremum problems with restrictions},
  author={Polyak, Boris Teodorovich},
  journal={Soviet Mathematics - Doklady},
  volume={166},
  pages={72-75},
  year={1966}
}

@article{rockafellar1976monotone,
author = {Rockafellar, R. Tyrrell},
title = {Monotone Operators and the Proximal Point Algorithm},
journal = {SIAM Journal on Control and Optimization},
volume = {14},
pages = {877-898},
year = {1976}
}

@book{Rockafellar09,
  title={Variational Analysis},
  author={Rockafellar, R. Tyrrell and Wets, Roger J-B.},
  year={2009},
  publisher={Springer Berlin, Heidelberg}
}

@article{Stella17,
author = {Stella, L. and Themelis, A. and Patrinos, P. },
title = { Forward–backward quasi-{N}ewton methods for nonsmooth optimization problems},
journal = { Computational Optimization and Applications },
volume = {67},
pages = {443-487},
year = {2017}
}

@incollection{themelis2019acceleration,
  author       = {Themelis, Andreas and Ahookhosh, Masoud and Patrinos, Panagiotis},
  title        = {On the acceleration of forward-backward splitting via an inexact {N}ewton method},
  booktitle    = {Splitting Algorithms, Modern Operator Theory, and Applications},
  editor       = {Heinz H. Bauschke and Regina S. Burachik and D. Russell Luke},
  pages        = {363--412},
  year         = {2019},
  publisher    = {Springer Cham}
}

@article{Themelis18,
  title={Forward-Backward Envelope for the Sum of Two Nonconvex Functions: Further Properties and Nonmonotone Linesearch Algorithms},
  author={Themelis, Andreas and Stella,  Lorenzo and  Patrinos, Panagiotis},
  journal={SIAM Journal on Optimization},
  volume={28},
  pages={2274-2303},
  year={2018}
}

@article{vlatakis2019efficiently,
  title={Efficiently avoiding saddle points with zero order methods: No gradients required},
  author={Vlatakis-Gkaragkounis, Emmanouil-Vasileios and Flokas, Lampros and Piliouras, Georgios},
  journal={Advances in neural information processing systems},
  volume={32},
  year={2019}
}

@article{VIAL1982187,
title = {Strong convexity of sets and functions},
journal = {Journal of Mathematical Economics},
volume = {9},
number = {1},
pages = {187-205},
year = {1982},
author = {Jean-Philippe Vial}
}

@incollection{Vidal_Zhu_Haeffele_2022,
  author      = "Vidal, René and Zhu, Zhihui and Haeffele, Benjamin D.",
  title       = "Optimization Landscape of Neural Networks",
  editor      = "Grohs, Philipp and Kutyniok, GittaEditors",
  booktitle   = "Mathematical Aspects of Deep Learning",
  publisher   = "Cambridge University Press",
  year        = 2022,
  pages       = "200–228",
}

@article{vladimirov1978uniformly,
  title={O ravnomerno 
 kvazivypuklyh funkcionalah [{O}n uniformly quasiconvex functionals]},
  author={Vladimirov, AA and Nesterov, Yu E and Chekanov, Yu N},
  journal={Vestnik Moskov. Univ. Ser. XV Vychisl. Mat. Kibernet},
  volume={4},
  pages={18--27},
  year={1978}
}

@article{WU2003Equivalence,
title = {Equivalence among various derivatives and subdifferentials of the distance function},
journal = {Journal of Mathematical Analysis and Applications},
volume = {282},
pages = {629-647},
year = {2003},
author = {Zili Wu and J.J. Ye}
}

@article{zaffaroni2004every,
  title={Is every radiant function the sum of quasiconvex functions?},
  author={Zaffaroni, Alberto},
  journal={Mathematical Methods of Operations Research},
  volume={59},
  pages={221--233},
  year={2004},
  publisher={Springer}
}

@article{zhu2024global,
  title={Global Convergence of High-Order Regularization Methods with Sums-of-Squares Taylor Models},
  author={Zhu, Wenqi and Cartis, Coralia},
  journal={arXiv preprint arXiv:2404.03035},
  year={2024}
}
\end{document}